\documentclass[10pt]{amsart} 
\usepackage[utf8]{inputenc}
\usepackage{amssymb,amsmath,amsfonts,color,graphicx,soul, multirow, mathabx, amsthm}
\usepackage[shortlabels]{enumitem}
\usepackage{graphicx} 
\usepackage{tikz}
\usetikzlibrary{calc}
\usetikzlibrary{patterns}
\usetikzlibrary{decorations.text}
\newtheorem{thrm}{Theorem}[section]

\newtheorem{lem}[thrm]{Lemma}
\newtheorem{prop}[thrm]{Proposition}
\newtheorem{cor}[thrm]{Corollary}
\theoremstyle{definition}

\newtheorem{deft}[thrm]{Definition}
\newtheorem{rem}[thrm]{Remark}

\newtheorem{quest}[thrm]{Open Questions}
\usepackage{fancyhdr}
\usepackage[unicode,colorlinks,
	citecolor=blue,
	linkcolor=blue,
	urlcolor=black,
]{hyperref}

\newcommand{\hs}{\text{ }}

\newcommand{\ol}{\overline}

\newcommand{\p}{\mathbb{P}}

\newcommand{\G}{\mathbb{G}}

\newcommand{\CB}{\mathcal{B}}
\newcommand{\phii}{\varphi}

\newcommand{\liff}{\leftrightarrow}

\newcommand{\rimp}{\rightarrow}
\newcommand{\dom}{\text{Dom}}

\newcommand{\proves}{\vdash}
\newcommand{\nproves}{\nvdash}

\renewcommand{\>}{\rangle}

\newcommand{\forces}{\Vdash}
\title{On Sets That Encode Themselves}
\author{Taeyoung Em}
\begin{document}
\begin{abstract}
Given partial information about a set, we are interested in fully recovering the original set from what is given. If a set encodes itself robustly, any partial information about the set suffices to fully recover the information about the original set. Jockusch defined a set $A$ to be introenumerable if each infinite subset of $A$ can enumerate $A$, and this is an example of a set which encodes itself. There are several other notions capturing self-encoding differently. We say $A$ is uniformly introenumerable if each infinite subset of $A$ can uniformly enumerate $A$, whereas $A$ is introreducible if each infinite subset of $A$ can compute $A$. We investigate properties of various notions of self-encoding and prove new results on their interactions. Greenberg, Harrison-Trainor, Patey, and Turetsky showed that we can always find some uniformity from an introenumerable set. We show that this can be reversed: we can construct an introenumerable set by patching up uniformity. This gives a rise to a new method of constructing a nontrivial introenumerable or introreducible set.
\end{abstract}
\maketitle
\section{Introduction}
For an infinite set $A\subseteq\omega$, denote $[A]^\omega$ to be the set of all infinite subsets of $A$, and $[A]^{<\omega}$ to be the set of all finite subsets of $A$. An infinite set $A$ is \emph{introreducible} if each $C\in[A]^\omega$ computes $A$. It is \emph{introenumerable}\footnote{Jockusch originally defined the notion of introenumerability in a different way: an infinite set $A$ is introenumerable if $A$ is $C$-c.e.\ for each $C\in[A]^\omega$. Nonetheless, the two definitions, in the uniform case, are equivalent by \cite[Proposition 2.1]{green}, and in the non-uniform case, they are equivalent by a similar argument with the help of Selman's Theorem.} if each $C\in[A]^\omega$ is enumeration\footnote{An enumeration operator $\Theta$ is a c.e.\ set of pairs $\<u,n\>$. For a set $A$, $\Theta(A)$ denotes the set $\{n:\exists u\hs D_u\subseteq A\wedge\<u,n\>\in\Theta\}$. We say $A\leq_e B$ if there is an enumeration operator $\Theta$ such that $\Theta(B)=A$. When $A\leq_e B$, we say $B$ is enumeration above $A$.} above $A$. It is \emph{majorreducible} if $A$ is computable by any function $f$ majorizing $p_A$, the principal function of $A$. Likewise, $A$ is \emph{majorenumerable} if each $f\geq p_A$ is enumeration above $A$. The notions have uniform versions, namely $A$ is \emph{uniformly introreducible} if each $C\in[A]^\omega$ computes $A$ in a uniform way, i.e., there is a Turing operator $\Phi$ such that $\Phi^C=A$ for $C\in[A]^\omega$. The other notions have their uniform versions similarly.

Constructing a nontrivial introenumerable set is a difficult task because unlike uniformly introenumerable sets where there is a single enumeration operator witnessing $\leq_e$, a nontrivial introenumerable set $A$ may have no such uniformity, and we have to ensure that \emph{every} infinite subset of $A$ is enumeration above $A$. Lachlan constructs an introreducible set which is not uniformly introreducible in \cite[Theorem 4.1]{jock}, and this seems to be the only place where constructing a nontrivial introenumerable/introreducible set is introduced in a literature. On the other hand, there has been attempts to find some uniformity from an introenumerable set. Greenberg, Harrison-Trainor, Patey, and Turetsky showed that if $A$ is introenumerable, there is a subset $C\in [A]^\omega$ and an enumeration operator $\Gamma$ such that $\Gamma(C_0)=A$ for all $C_0\in [C]^\omega$ \cite[Proposition 3.7]{green}. We show that this can be reversed, which gives a rise to a new method of building a nontrivial introenumerable set.

\subsection{Background} Dekker and Myhill introduced first-order versions of the notions capturing self-encoding. Indeed, retraceability was the main theme of Dekker and Myhill's paper, where they defined introreducibility as a property of retraceable sets, although introreducibility is in itself very interesting. An infinite set $A$ is \emph{regressive} if $A$ can be written as $A=\{a_0,a_1,\dots\}$, where $a_i$'s are distinct and there exists a partial computable function $\phii$ with $\phii(a_{n+1})=a_n$ for all $n\geq 0$ and $\phii(a_0)=a_0$. It is \emph{retraceable} if such $a_i$'s can be chosen in the increasing way, $a_0<a_1<\cdots$. Retraceable sets have been studied carefully in \cite{dekker} where they proved that every c.e.\ degree has a co-c.e.\ retraceable representative, and that every degree has a retraceable representative. Mansfield (unpublished) proved that complementary\footnote{Let $P$ be a property such as regressive. A set $A$ is complementary $P$ if $A$ and $\ol{A}$ are both $P$.} retraceable sets are computable. Regressive sets have been studied carefully by Appel and McLaughlin in \cite{appel} where they proved that if $A$ and $\ol{A}$ are complementary regressive sets, then one of them is c.e. The Mansfield's result can follow as a corollary to it. We note that there does exist noncomputable complementary regressive sets, unlike retraceable sets. Any co-c.e.\ retraceable set is complementary regressive, and every c.e.\ degree has such a set.

There are some easy implications between these notions which capture self-encoding. Clearly, every retraceable set is uniformly introreducible, and every regressive set is uniformly introenumerable. Furthermore, every (uniformly) majorreducible set is (uniformly) introreducible, as any subset's principal function is sparser than the original set's principal function. Interestingly, the relation between regressive and retraceable sets or the relation between introenumerable and introreducible sets is reminiscent of the relation between c.e.\ and computable sets. Indeed, every regressive set has a retraceable subset \cite{dekker2}, and every uniformly introenumerable set has a uniformly introreducible subset \cite[Theorem 1.4]{green}. These can be thought of as analogues of the fact that every infinite c.e.\ set has an infinite computable subset. 

\subsection{Motivation} Knowing these, it is natural to seek for a similarity between retraceable sets and introreducible sets, and a similarity between regressive sets and introenumerable sets. As retraceable sets do, complementary introreducible sets are computable as shown in \cite[Theorem 2.19]{slaman}. This is an interesting property of introreducible sets, as this does not happen even for uniformly introenumerable sets \cite[Theorem 5.2]{jock}, nor regressive sets. In this sense, retraceable sets share some properties with introreducible sets, and regressive sets share some properties with introenumerable sets. But certainly, retraceability is distinct from introreducibility, and regressive is distinct from introenumerability. If $A$ and $\ol{A}$ are immune semicomputable sets, then $A$ and $\ol{A}$ are hyperimmune uniformly introenumerable \cite[Theorem 5.2]{jock}. So unlike regressive sets, uniformly introenumerable sets can be complementary immune. And we show in this paper several instances where retraceability and introreducibility fall apart.

\subsection{Main results of this paper} Jockusch left a number of interesting open questions in his paper \cite{jock}, and some of them are answered in \cite{green} by Greenberg, Harrison-Trainor, Patey, and Turetsky. In this paper, we address some open questions from \cite{jock} and \cite{green}, improve some of their results, then do a careful analysis on how (uniformly) intro-e/T\footnote{`intro-e/T' is an abbreviation of `introenumerable/introreducible'. Likewise, we sometimes write `major-e' for `majorenumerable' and `major-T' for `majorreducible'.}, (uniformly) major-e/T, regressive/retraceable, and c.e./computable sets coincide/fall far apart from each other, and study some of their properties. Thus, the main focuses of this paper are:

Let $P,Q$ be two distinct classes among class of (uniformly) intro-e/T, (uniformly) major-e/T, regressive/retraceable, and c.e./computable sets.
\begin{enumerate}
\item What are the properties of sets in $P$?

\item Whether every set in $P$ has a subset in $Q$?

\item Whether every set in $P$ is also in $Q$, and if not, in what circumstances, do they coincide?
\end{enumerate}
The first and second items are studied across all sections of this paper, and the last section mainly focuses on the second and the third items. In Section 3, we focus on regressive and retraceable sets and how they relate to majorreducible sets. In Section 4, we focus on introenumerable and introreducible sets, where we also introduce a new method of building an introenumerable set. The open questions of this paper are addressed in Section 4 Open Questions \ref{open questions}.
\begin{figure}[!ht]
\centering
\begin{tikzpicture}[scale=2.5, line join=round]

  \fill[gray!30] (1, 0) rectangle (4, 1);

  \draw[thick] (0, 0) grid (4, 2);

  \draw[line width=3pt] (2,0) -- (4,0);
  \draw[line width=3pt] (4,0) -- (4,1);
  \draw[line width=3pt] (3,1) -- (4,1);
  \draw[line width=3pt] (2,0) -- (3,1);

  \draw[thin] (-0.2, 0) -- (-0.2, 0.25);
  \draw[thin] (-0.2, 0.75) -- (-0.2, 1);
  \draw[thin] (0, 0) -- (-0.2, 0);
  \draw[thin] (0, 1) -- (-0.2, 1);
  \node[rotate=90, anchor=south] at (-0.13, 0.5) {intro-T};

  \draw[thin] (1,2) -- (1,2.2);
  \draw[thin] (1,2.2) -- (2.12,2.2);
  \draw[thin] (2.88,2.2) -- (4,2.2);
  \draw[thin] (4,2.2) -- (4,2);
  \node[anchor=south] at (2.5, 2.12) {unif intro-e};

  \draw[thin] (2,0) -- (2,-0.2);
  \draw[thin] (2,-0.2) -- (2.65,-0.2);
  \draw[thin] (3.35,-0.2) -- (4,-0.2);
  \draw[thin] (4,-0.2) -- (4,0);
  \node[anchor=north] at (3, -0.1) {regressive};

  \draw[thin] (3,2) -- (3.15,2.4);
  \draw[thin] (3.15,2.4) -- (3.35,2.4);
  \draw[thin] (3.65,2.4) -- (3.85,2.4);
  \draw[thin] (3.85,2.4) -- (4,2);
  \node[anchor=north] at (3.5,2.48) {c.e.};
  \node at (2.75,0.35) {\footnotesize $\bullet$\cite{dekker}};
  \node at (2.4,0.7) {\footnotesize $\bullet$\ref{regressive intro T}};
  \node at (2.6,1.5) {\footnotesize $\bullet$\ref{regressive not intro T}};
  \node at (1.6,1.5) {\footnotesize $\bullet$\ref{neither}};
  \node at (0.6,1.6) {\footnotesize $\bullet$\ref{ienotuie forcing proof}};
  \node at (0.6,1.4) {\footnotesize $\bullet$\ref{intro e neither}};
  \node at (0.6,0.5) {\footnotesize $\bullet$\ref{unif intro e vs intro T}};
  \node at (1.6,0.5) {\footnotesize $\bullet$\ref{remark}};
  \node at (3.5, 0.5) {\footnotesize Computable};
  \node[align=center] at (3.5,1.5) {\footnotesize Noncomputable\\c.e.};
\end{tikzpicture}
\caption{The universe of introenumerable sets. The bottom four boxes (to the right of `intro-T' label) are the introreducible sets, and the right six boxes (under the `unif intro-e' label) are the uniformly introenumerable sets. Similarly, the four right-most boxes are the regressive sets, and the two right-most boxes are the c.e.\ sets. The intersection of the uniformly introenumerable boxes and the introreducible boxes is colored in gray, and it is exactly for the uniformly introreducible sets. The trapezoid with the thick lines is for the retraceable sets.}
\label{figure1}
\end{figure}
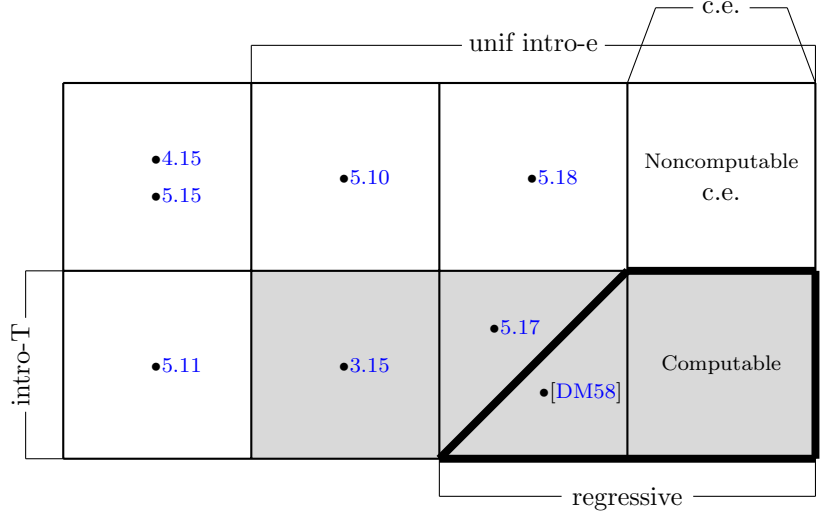

Figure \ref{figure1}---which is the universe of introenumerable sets---illustrates item (3). For example, the right triangle inside the trapezoid is for the retraceable, uniformly introreducible, but not c.e.\ sets, whereas the right triangle outside the trapezoid is for regressive, uniformly introreducible, but not retraceable sets. The box at the top-left corner is for the introenumerable sets which are neither uniformly introenumerable nor introreducible. In each area of the figure, we have indicated where it is proved that the area is nonempty. A similar picture is true for (uniformly) major-e/T sets which is proved in the last section, but is omitted in the figure for visibility.

For two notions $P$, $Q$, we write $P\subseteq Q$ and say $P$ is a subset of $Q$ if every set in $P$ is in $Q$. We say $P$ is \emph{refinable} to $Q$ if every set in $P$ has a subset which is in $Q$. We say two notions $P$, $Q$ are \emph{distinct} if there is a set in one of them which is not in the other. Hence, together with Figure \ref{figure1}, we establish:
 \begin{thrm}\label{separation}
 Uniformly major-T/e, major-T/e, uniformly intro-T/e, intro-T/e, regressive, retraceable, c.e., and computable are all distinct notions. Any two notions $P$, $Q$ from the following four are refinable to each other.
 $$\text{major-T/e and uniformly major-T/e}.$$
 \end{thrm}
 The last assertion is proved at the end of Section 2. That the notions uniformly major-T/e and major-T/e are distinct is established at the beginning of Section 5. We indicate below where it is proved for those that do not follow from the figure above.
 \begin{itemize}
 \item A uniformly majorenumerable set which is not introreducible: by Corollary \ref{ume not umt}.
 \item A uniformly majorreducible set which is not regressive: by Remark \ref{remark}.
 \item A retraceable set which is not majorenumerable: by Corollary \ref{retraceable not me}.
 \end{itemize}
\subsection{Binary relations}
It is useful to define binary relations for the notions capturing self-encoding. Most of the binary relations defined here, excluding those about major-T/e, are introduced in \cite{green}.
\begin{deft}
For infinite sets $A,B\subseteq\omega$, we say $A\leq^i_T B$ if each $C\in[B]^\omega$ computes $A$, read as $B$ \emph{introcomputes} $A$. Likewise, we can define the uniform version, $A\leq^{ui}_T B$. Of course, this can be defined for other reductions than the Turing reduction: $A\leq^i_e B$ if each $C\in [B]^\omega$ is enumeration above $A$; $A\leq^m_T B$ if any $f\geq p_B$ computes $A$, taking the oracle as the graph of $f$; $A\leq^m_e B$ if any $f\geq p_B$ is enumeration above $A$. And we define their uniform versions similarly.
\end{deft}
It is shown in \cite[Proposition 2.1]{green} that $\leq_{e}^{ui}$ and $\leq_{c.e.}^{ui}$ coincide, and a similar argument (with Selman's Theorem) shows that $\leq_e^i$ and $\leq_{c.e.}^i$ coincide. Likewise, a similar argument shows that $\leq^{um}_e(\leq^m_e)$ coincides with $\leq^{um}_{c.e.}(\leq^m_{c.e.})$.
\section{Majorenumerable and majorreducible sets}
We show that every majorenumerable set has a uniformly majorreducible subset in this section. It is Solovay's result in \cite{solov} that hyperarithmetic sets are exactly those with a uniform modulus. From his results, we get that $B$ is hyperarithmetic iff $B\leq_T 0^{(\beta)}$ for some $\beta<\omega_1^{ck}$ and there is some $f\equiv_T 0^{(\beta)}$ such that $B\leq^{um}_T f$ iff there is some $f$ such that $B\leq^{um}_T f$. The below theorem is a sharpened version of \cite[Proposition 2.3]{green} that a set $A$ is $\Pi_1^1(\text{or } \Sigma_1^1)$ if and only if there is $f$ such that $A\leq^{um}_{\text{c.e.}} (\text{or }\leq^{um}_{\text{co-c.e.}}\text{ resp.}) f$.

In the below proof and later proofs, we use a method of forcing, where we use a term \emph{sufficiently generic}. On a forcing notion $\p$ and a filter $\G$ on $\p$, $\G$ is a sufficiently generic filter if it meets a countable collection of dense sets that we are considering---a collection $(D_i,i\in\omega)$ which should be clear from each context. We say a set $D\subseteq \p$ is \emph{dense above} $\G$ if for each $p\in \G$, there is $q\leq_\p p$ with $q\in D$. Let $R$ be a property that we want our sufficiently generic filter to posses. It suffices to show that $\{p\in\p:p\forces R\}$ is dense above $\G$ because $\{p\in\p:p\forces R\vee p\forces\neg R\}$ is a dense set, and $\G\cap\{p\in\p:p\forces\neg R\}$ is empty when $\{p\in\p:p\forces R\}$ is dense above $\G$.
\begin{thrm}\label{pi11}
Let $A$ be an infinite set. The following are equivalent.
\begin{itemize}
\item $A$ is $\Pi_1^1(\Sigma^1_1)$.
\item There is $f$ such that $A\leq^{um}_{\text{c.e.}} (\leq^{um}_{\text{co-c.e.}}) f$.
\item There is $f$ such that $A\leq^{m}_{\text{c.e.}} (\leq^{m}_{\text{co-c.e.}}) f$.
\end{itemize}
\end{thrm}
\begin{proof}
It suffices to show the $\Pi^1_1$ case. The first and second items' equivalence is due to \cite[Proposition 2.3]{green}. And the implication from the second to the third item is obvious. 

For the implication from the third to the first item, assume there is $f$ such that $A\leq^m_{c.e.}f$. Say a string $\sigma\in \omega^{<\omega}$ is above a function $g$, denoted by $\sigma\geq g$, if $\sigma(n)\geq g(n)$ for all $n$ where $\sigma$ is defined. We consider the following forcing notion $\p$ which is a form of Hechler forcing. $\p$ consists of the pairs $(\sigma,g)$ such that $\sigma\geq g\geq f$. Say $(\sigma_2,g_2)$ extends $(\sigma_1,g_1)$ if $\sigma_2\succ\sigma_1$ and $g_2\geq g_1$. Let $\G$ be a sufficiently generic filter on $\p$, and let $G=\bigcup\{\sigma:\exists p\in\G\hs(p=(\sigma, g)\text{ for some }g)\}$. It is clear that $G$ is a total function majorizing $f$. From that $A\leq^m_{c.e.} f$, there is some $e$ such that $W_e^G=A$. Then by the sufficient genericity, it must be that for some $(\sigma,g)\in\G$, $(\sigma,g)\forces W_e^G=A$---as otherwise, $\{(\sigma,g)\in\p:(\sigma,g)\forces W_e^G\neq A\}$ would be dense above $\G$. This means:

\begin{enumerate}
\item For every $m\in A$ and every $h$ such that $\sigma\geq h\geq g$, there exists $\tau$ such that $\tau\succ \sigma$ and $\tau\geq h$ with $m\in W_e^\tau$.

\item For every $m$ and every $(\tau,h)$ extending $(\sigma,g)$ in $\p$, if $m\in W_e^\tau$, then $m\in A$.
\end{enumerate}
From these, we obtain a $\Pi^1_1$ description of $A$:
$$\forall n[n\in A\liff\forall h\leq\sigma\hs\exists\tau(\tau\succ\sigma\wedge\tau\geq h)\hs n\in W_e^\tau].$$
\end{proof}
Every majorreducible set has a modulus. Note that having a modulus implies that it is hyperarithmetic.
\begin{cor}[Solovay]\label{hyp maj}
An infinite set $A$ is hyperarithmetic iff there is $f$ with $A\leq^m_T f$ iff there is $f$ with $A\leq_T^{um}f$. In particular, every majorreducible set is hyperarithmetic.
\end{cor}
Jockusch proved that every infinite $\Pi^1_1$ set has a uniformly majorreducible subset by noting that every infinite $\Pi^1_1$ set has an infinite hyperarithmetic subset. The converse is true as well.
\begin{cor}[Jockusch, Solovay]\label{jock sol}
$A$ has a uniformly majorreducible subset iff $A$ has a majorreducible subset iff $A$ has a $\Delta^1_1$ subset.
\end{cor}
\begin{proof}
If $A$ has a majorreducible subset $B$, then by Corollary \ref{hyp maj}, $B$ is hyperarithmetic. Conversely, if $A$ has a $\Delta_1^1$ subset $B$, then $A$ has a uniformly majorreducible subset \cite[Corollary 6.9]{jock}.
\end{proof}
\begin{cor}\label{majorpi11}
Every majorenumerable set is $\Pi_1^1$.
\end{cor}
The following proves the last assertion of Theorem \ref{separation}.
\begin{cor}\label{refinable}
Every majorenumerable set has a uniformly majorreducible subset.
\end{cor}
\begin{proof}
Noting the Luckham's result that every infinite $\Pi_1^1$ set has an infinite hyperarithmetic subset, this is immediate from Corollary \ref{jock sol}. 
\end{proof}
We will see that retraceability is quite distinct from uniform majorreducibility. Nevertheless, we show that under some conditions, they coincide.
\section{Regressive and retraceable sets}
We focus on regressive and retraceable sets on this section. We show that every co-immune regressive set with a total function witnessing the regressiveness is uniformly majorreducible, and that the assumption of having such total function is necessary.
\begin{deft}
A partial function $\phii_e$ is a \emph{regressing function} for a regressive set $A=\{a_0,a_1,\dots\}$ if $\phii_e(a_{n+1})=a_n$ for all $n\geq 0$ and $\phii_e(a_0)=a_0$. Further, $\phii_e$ is a \emph{retracing function} for $A$ if $A$ is retraceable.
\end{deft}
\begin{deft}
Let $A=\{a_0,a_1,\dots\}$ be a regressive/retraceable set witnessed by $f$. We say $f$ is \emph{special} if $a_0\in \hat{f}\footnote{Throughout this paper, for a function $f$, $\hat{f}$ refers to this set.}(a):=\{x:\exists n \hs f^n(a)=x\}$ for all $a\in\dom(f)$, where $f^n$ denotes $f$ applied $n$ times.
\end{deft}
\begin{lem}\label{specialfunction}
Every regressive/retraceable set has a special regressing/retracing function.
\end{lem}
\begin{proof}
If $A=\{a_0,a_1,\dots\}$ is regressed by $f$, then we can define a partial computable function $g$ so that $g$ halts on $n$ if and only if $a_0\in \hat{f}(n)$.
\end{proof}
So from now on, unless otherwise specified, we will assume that every regressing/retracing function is special.
\begin{deft}
For a regressive/retraceable set $A=\{a_0,a_1,\dots\}$ witnessed by $f$, a c.e.\ tree $T_{f,A}\subseteq \omega^{<\omega}$ is a \emph{regressing/retracing tree} if $(b_0,b_1,\dots,b_n)\in T_{f,A}$ if and only if $f(b_n)=b_{n-1},\dots,f(b_1)=b_0$ and $b_0=a_0$. We say $b\in\omega$ is \emph{on} the tree $T_{f,A}$ if there is some node $\sigma\in T_{f,A}$ such that $\sigma(i)=b$ for some $i<|\sigma|$.
\end{deft}

Note that every regressive/retraceable set has a regressing/retracing tree.
\subsection{Co-immune regressive/retraceable sets}
If a regressive set has a c.e.\ subset, then it is c.e., and likewise, if a retraceable set has a c.e.\ subset, then it is computable. In this sense, immune regressive or immune retraceable sets are the interesting ones. On the other hand, it turns out that regressive/retraceable sets are easier to analyze if their complements have no c.e.\ subset.
\begin{prop}\label{regressingtree}
If $A$ is co-immune regressive/retraceable, then its regressing/retracing tree is a finitely branching tree with a unique infinite path, which is $A$ itself.
\end{prop}
\begin{proof}
If $T_{f,A}$ has a node $\sigma$ with infinitely many successor nodes $\sigma^\frown b_0,\sigma^\frown b_1,\dots\in T_{f,A}$, then at most one of $b_0,b_1,\dots$ is a member of $A$. Since we can computably enumerate the $b_i$'s, there should not be infinitely many $b_i$'s by the co-immunity. Likewise, if there is a node $\sigma\in T_{f,A}$ such that has infinitely many nodes $\tau\in T_{f,A}$ with $\tau\succ\sigma$, then by the co-immunity, it must be that $A\succ\sigma$. In particular, if $B\in [T_{f,A}]$, then $B=A$.
\end{proof}
\begin{prop}[McLaughlin]\label{0'comp}
If $A=\{a_0,a_1,\dots\}$ is a co-immune regressive set, then $A$ is $0'$-computable.
\end{prop}
\begin{proof}
Let $f$ regress $A$ and let $T_{f,A}$ be its regressing tree. We first show that $0'$ can list the $a_i$'s. Having $a_n$, let $\sigma=(a_0,\dots,a_n)$. Due to the co-immunity, there are only finitely many numbers $x\in\omega$ such that $f(x)=a_n$. Hence, $0'$ can list all numbers $x\in\omega$ such that $f(x)=a_n$ by keep on asking ``are there more numbers such that $f(x)=a_n$?'' and listing such $x$ until the answer is ``no''. Among the numbers in the list, $0'$ can eliminate all $x$ such that $\sigma^\frown x\in T_{f,A}$ has only finitely many nodes extending it on the tree, until precisely one number is left in the list.

Say the remaining number is $x_0$. As in the proof of Proposition \ref{regressingtree}, we have $x=a_{n+1}$ if and only if $\sigma^\frown x$ has infinitely many nodes extending it on $T_{f,A}$. And it must be that $x_0=a_{n+1}$ because the list has to contain a number which equals to $a_{n+1}$.

To list elements of $\ol{A}$, enumerate $a$ into $\ol{A}$ if and only if either $a$ is not on the tree $T_{f,A}$ or there is a node $\sigma\in T_{f,A}$ which has only finitely many nodes extending it on the tree and $\sigma(|\sigma|-1)=a$.
\end{proof}
It is shown in \cite[Corollary 3.1]{appel} that not every regressive set has a total regressing function,\footnote{A total regressing function is a regressing function which is also total.} by showing that no simple set can possess a total regressing function---note that every c.e.\ set is regressive as c.e.\ sets have a computable enumeration $a_0,a_1,\dots$. Thus, it is natural to ask if this also happens for retraceable sets. The answer is positive, and we can further require the set to be co-immune.
\begin{prop}\label{co-immune req}
There exists a co-immune retraceable set $\leq_T 0'$ that has no total retracing function.
\end{prop}
\begin{proof}
We build a partial computable function $f$ that will be a special retracing function for a co-immune retraceable set $A$. At the same time, we ensure that $A$ cannot have a total retracing function. The domain of $f$ will be a c.e.\ tree with a unique infinite path, which will be $A$, as in Proposition \ref{regressingtree}. If we succeed in building such $f$ and $A$, then $A\leq_T 0'$ is automatic from Proposition \ref{0'comp}. The idea is to diagonalize against all total computable functions, while assuring that $f$ will still retrace some set. We do this by a finite injury argument with the requirements:
\begin{align*}
P_e&: \text{if }\phii_e\text{ is total, then }\phii_e\text{ is not a retracing function for }A,\\
Q_e&: \text{if }W_e\text{ is infinite, then }A\cap W_e\neq\emptyset.
\end{align*}
At stage $s+1$, each requirement with $e\leq s$ will be \emph{connected} to a number. Let $p_{e,s}$ denote the number that $P_e$ is connected to, and let $q_{e,s}$ denote the number that $Q_e$ is connected to at stage $s+1$. Further, at stage $s+1$, each $P_e$ requirement with $e\leq s$ will hold a witness $w_{e,s}$. And whenever a $Q_e$ requirement acts and becomes satisfied at stage $s$, it will be assigned a witness $x_{e,s}$, which is kept as a witness of $Q_e$ as long as $Q_e$ remains satisfied. The requirements' priorities are $P_0>Q_0>P_1>Q_1>\cdots$.
\vspace{0.3cm}

\noindent\emph{Construction.} At stage $0$, let $w_{0,0}=2$ and $p_{0,0}=0=q_{0,0}$. We set $f_0(0)=0$. At stage $s+1$, search for the highest priority requirement with $e\leq s$ that requires attention and is unsatisfied, then let it act as the below. Otherwise, if there is no such requirement, then go to the last paragraph of the construction.

\emph{How they act.} We say $P_e$ requires attention at stage $s+1$ if $e\leq s$ and $\phii_{e,s}(w_{e,s})\downarrow$. If $P_e$ is the highest priority requirement requiring attention at stage $s+1$, we let $P_e$ act by:
\begin{enumerate}
\item If $\phii_{e,s}(w_{e,s})\neq p_{e,s}$, then set $f_{s+1}(w_{e,s})=p_{e,s}$.
\item If $\phii_{e,s}(w_{e,s})=p_{e,s}$, then set $f_{s+1}(w_{e,s})=w_{e,s}-1$. And set $f_{s+1}(w_{e,s}-1)=p_{e,s}$.
\end{enumerate}
In this case, each lower priority requirement is injured, being declared unsatisfied if it were satisfied. We set $w_{e+1,s+1}$ as the smallest even number larger than any number mentioned so far. Then we set $w_{d,s+1}$ for $e+1<d\leq s+1$ as $w_{e+1,s+1}+2(e+1-d)$. Let each $P_d$ requirement with $e<d\leq s+1$ be connected to $w_{e,s}$ by setting $p_{d,s+1}=w_{e,s}$. For $d\geq e$, $x_{d,s+1}$ is undefined since the lower priority requirements are injured. Set $q_{d,s+1}=w_{e,s}$ for $e\leq d\leq s+1$.

Next, we consider the $Q$ requirements. At stage $s+1$, we say a number $x\in\dom(f_{s})$ \emph{respects priorities up to $P_e$} if
\begin{itemize}
\item $0\in\hat{f}_s(x)$,
\item $\hat{f}_s(x)$ contains $w_{d,s}$ for each $d$ such that $P_d$ is satisfied now and $d\leq e$, and
\item $\hat{f}_s(x)$ contains $x_{d,s}$ for each $d$ such that $Q_d$ is satisfied now and $d<e$.
\end{itemize}
Informally, this means that putting $x\in A$---or believing that $x$ is on the unique infinite path of $T_{f,A}$---does not conflict with the higher priority requirements which are satisfied. We say $Q_e$ requires attention at stage $s+1$ if $e\leq s$ and there is some $m\in W_{e,s}$ such that either
\begin{enumerate}
\item $m\notin \dom(f_s)$ and $m>w_{e,s}$, or
\item $m\in\dom(f_s)$ and $m$ respects priorities up to $P_e$.
\end{enumerate}
Assume $Q_e$ is the highest priority requirement requiring attention at stage $s+1$. In case $(1)$, we let $Q_e$ act by setting $x_{e,s}=m$ and $f_{s+1}(m)=q_{e,s}$. In case $(2)$, we let $Q_e$ act by just setting $x_{e,s}=m$. In either cases, all the lower priority requirements are injured, and we set $w_{e+1,s+1}$ as the \emph{second} smallest even number larger than any number mentioned so far. Then we set $w_{d,s+1}$ for $e+1<d\leq s+1$ as $w_{e+1,s+1}+2(e+1-d)$. Set $q_{d,s+1}=p_{d,s+1}=x_{e,s+1}$ for $e<d\leq s+1$.

Finally, to end stage $s+1$ construction, for each $d\leq s$, if $p_{d,s+1},q_{d,s+1}, w_{d,s+1}$, or $x_{d,s+1}$ (if $x_{d,s}$ is defined) is yet undefined, define them identically to their previous value, e.g.\ $p_{d,s+1}=p_{d,s}$ if $p_{d,s+1}$ is undefined so far. Likewise, for elements of $\dom(f_s)$ that is not yet defined on $f_{s+1}$, define them on $f_{s+1}$ so that $f_{s+1}$ extends $f_s$. If no requirement required attention, define $p_{s+1,s+1}=p_{s,s}$, $q_{s+1,s+1}=q_{s,s}$, and $w_{s+1,s+1}=w_{s,s}+2$.
\vspace{0.3cm}

\noindent\emph{Verification.} It is clear that the function $f=\bigcup_s f_s$ we build is partial computable. To check that the construction is well-defined, we only need to check that $w_{e,s}$ and $w_{e,s}-1$ are not in $\dom(f_s)$ at the beginning of stage $s+1$. It cannot be that $w_{e,s}\in \dom(f_s)$ by an action of a $P$ requirement: if $d>e$ and $P_d$ acts, this only puts elements $>$ $w_{e,s}$ into $\dom(f_s)$, and if $d<e$ and $P_d$ acts, this action sets $w_{e,s}$ to be some even number larger than whatever $P_d$ puts into $\dom(f_s)$. And $w_{e,s}$ cannot be put into $\dom(f_s)$ via a $Q$ requirement either. If $Q_d$ with $d\geq e$ acts according to case $(1)$, then it can only put numbers larger than $w_{e,s}$ into $\dom(f_s)$. If $Q_d$ with $d<e$ acts according to case $(1)$, then the action sets $w_{e,s}$ to be some even number larger than whatever $Q_d$ puts into $\dom(f_s)$. And we do not have to worry about a $Q$ requirement's action according to case $(2)$---say at stage $t<s+1$---since we can assume using an induction that $w_{e,s}\notin \dom(f_t)$; note that the action of a $Q$ requirement according to case $(2)$ does not change the domain of $f$. We can similarly show that $w_{e,s}-1$ is not in $\dom(f_s)$ at the beginning of stage $s+1$.

\emph{The function $f$ retraces a unique infinite set.} We first show that whenever $W_e$ is an infinite c.e.\ set, $Q_e$ becomes satisfied perpetually from some stage. Assume $W_e$ is infinite, and let $s$ be a large enough stage so that no requirement of a higher priority than $Q_e$ acts from then on. Let $N=\max(\{w_{e,s}\}\cup\dom(f_s))$. Then there exists a number $z>N$ such that $z\in W_e\setminus W_{e,s}$. Say $z$ joins $W_e$ at stage $t>s$. If $Q_e$ is already satisfied at stage $t$, then it cannot be injured from then on. Otherwise, $Q_e$ requires attention at stage $t+1$ because of the choices of $N$ and $z$. Indeed, we have $N>w_{e,t}=w_{e,s}$, and any number in the set $\dom(f_t)\setminus\dom(f_s)$ respects priorities up to $P_e$. So $Q_e$ becomes satisfied perpetually from stage $t+1$, and $x_e:=\lim_s x_{e,s}$ exists.

Noting that there are infinitely many infinite c.e.\ sets, there exist infinitely many $e$'s such that $Q_e$ is satisfied perpetually, and hence $x_e$ is defined. Moreover, if $x_e$ and $x_d$ are defined and $e<d$, then $x_e\in\hat{f}(x_e)$. To see this, say $t_e$ and $t_d$ are the least stages such that $Q_e$ and $Q_d$ (resp.) become satisfied perpetually. Since all the lower priority requirements are injured after $Q_e$ acts, it must be that $t_e<t_d$. Hence, at stage $t_d$, no matter if the action was according to case $(1)$ or $(2)$, we have $x_{e}\in\hat{f}_{t_d+1}(x_{d})$. Let $E$ be the set of all indices $e$ such that $W_e$ is an infinite c.e.\ set. Then $A:=\bigcup_{e\in E}\hat{f}(x_e)$ is the unique infinite path of $\dom(f)$ viewed as a tree. The fact that $f$ retraces $A$ can be shown using an induction on stages.

\emph{$A$ is co-immune.} Assume that $W_e\subseteq \ol{A}$ is an infinite c.e.\ set. But as we just showed above, $Q_e$ then becomes perpetually satisfied from some stage. This means that an element of $W_e$ must be an element of $A$, which is a contradiction.

\emph{There is no total retracing function for $A$.} Assume that $\phii_e$ is a total computable function retracing $A$. Let $s$ be a large enough stage such that no requirement of a higher priority than $P_e$ acts from then on. Then note that $w_{e,t}=w_{e,s}$ and $p_{e,t}=p_{e,s}$ for all $t>s$. Let $t>s$ be such that $\phii_{e,t}(w_{e,t})\downarrow$. If $P_e$ is already satisfied at stage $t$, then $P_e$ succeeded in a diagonalization and never gets injured from then on, so we get a contradiction. If $P_e$ is not yet satisfied at stage $t$, then $P_e$ is the highest priority requirement requiring attention at stage $t+1$. So $P_e$ acts at stage $t+1$, again diagonalizing against $\phii_e$ as a contradiction.
\end{proof}
Despite Lemma \ref{specialfunction}, when $A$ is a regressive set which has a total regressing function $f$, this does not necessarily imply that $A$ has a special regressing function which is also total, as making $f$ special may delete some of its elements from the domain---see below Proposition \ref{counterex3}. Nonetheless, the below propositions show that for retraceable sets, having a total retracing function implies having a special total retracing function, and for co-immune regressive sets, having a total regressing function implies having a total regressing function.
\begin{prop}
Let $A$ be a retraceable set. If there exists a total computable function retracing $A$, then there exists a special total computable function which retraces $A$.
\end{prop}
\begin{proof}
Let $f$ be a total computable function which retraces $A$. Without loss of generality, assume that $0\in A$, so $0\in\hat{f}(a)$ for all $a\in A$. We define a total computable function $g$ as follows: Let $g(0):=0$. Assume we have defined $g(0),\dots,g(n)$. If $f(n+1)\geq n+1$, then we let $g(n+1):=0$. Otherwise, if $f(n+1)< n+1$, then we let $g(n+1):=f(n+1)$.

Clearly, $g$ is a total computable function. And $g$ retraces $A$ as $g(m)=f(m)$ for all $m\in A$. Lastly, $g$ is special because for $n\notin A$, there must exist some $m\in\omega$ such that $f^{m+1}(n)\geq f^{m}(n)$. Let $m^*$ be the least such $m$. Then $g(f^{m^*}(n))=0=g^{m^*+1}(n)$, so that $0\in \hat{g}(n)$.
\end{proof}
\begin{prop}\label{totaltree}
If $A=\{a_0,a_1,\dots\}$ is a co-immune regressive set which has a total regressing function (here, the regressing function is not assumed to be special), then $A$ has a special total regressing function.
\end{prop}
\begin{proof}
Let $f$ be a (not necessarily special) total regressing function of $A$, and consider its regressing tree $T_{f,A}$. For elements $b$ not on some node of the tree (i.e.\ the elements $b$ such that $a_0\notin \hat{f}(b)$), if $\hat{f}(b)$ is infinite, then we get an infinite c.e.\ subset of $\ol{A}$. So $\hat{f}(b)$ is finite for all such $b$. The only possible situation for such $b$ is that $\hat{f}(b)$ contains some $m$-cycle, meaning that $\hat{f}(b)$ contains an element $c$ such that $f^m(c)=c$ for some $m$.

Moreover, we can computably enumerate the elements $b$ such that $\hat{f}(b)$ contains some $m$-cycle and $a_0\notin \hat{f}(b)$. So again by the co-immunity, there are only finitely many such elements. So there are only finitely many $b$'s such that $a_0\notin \hat{f}(b)$. Hence, $\dom(f)=\omega$ has only finitely many elements not on some node of the tree $T_{f,A}$. So by connecting those finitely many elements not on $T_{f,A}$ directly to $a_0$, we may take a partial computable function $g$ which is a special total function, so that every element of $\omega$ appears on some node of the tree $T_{g,A}$.
\end{proof}
\begin{prop}\label{counterex3}
There exists a regressive set $A$ that has a total regressing function, but does not have a special total regressing function.
\end{prop}
\begin{proof}
We claim that $A=0'$ is such a regressive set. For this, we build a total computable function $f=\bigcup_s f_s$ by stages that will witness that $A$ is a regressive set. In the construction, at each stage $s+1$, if we are unsure whether $\phii_s(s)$ will converge or not, we first let $s$ be mapped to some appropriate index $S\in\omega$ such that $s\in 0'$ if and only if $S\in 0'$. If later, we know that $\phii_s(s)$ converges, that shall mean that $s$, together with all the numbers connected to $s$ via $f$, should be in $0'$. So in that case, we let them be in $A$ by appropriately defining them as some $a_i$'s. 
\vspace{0.3cm}

\noindent\emph{Construction.} Without loss of generality, assume that $0\in 0'$. We set $f_0=\{(0,0)\}$, $a_0=0$, and $A_0=\{a_0\}$. At stage $s+1$, assume that $\{0,\dots,s\}\subseteq \dom(f_s)$, and $a_0,\dots,a_m$ are defined for some $m\geq s$; so we have $A_s=\{a_0,\dots,a_m\}$ so far.

\emph{Step 1.} If $s+1$ is already in $\dom(f_s)$, then go to \emph{Step 2.} Otherwise, if $s+1\notin \dom(f_s)$, let $S$ be a large number such that
\begin{enumerate}
\item $S$ has not been mentioned so far and
\item $\phii_S$ is such that for all $n\in\omega$, $\phii_S(n)$ converges if and only if $\phii_{s+1}(s+1)$ converges.
\end{enumerate}
Define $f_{s}(s+1):=S$, so that $s+1$ is in $\dom(f_s)$ now. Go to \emph{Step 2.} below and follow the instruction.

\emph{Step 2.} Now that $s+1\in \dom(f_s)$, search for a large number $N$ that has not been mentioned so far and $\phii_N(N)\downarrow$. Set $a_{m+1}=N$. Further, for each number $n\in\dom(f_s)$ such that $n\notin A_s$, check if $\phii_{n,s}(n)\downarrow$. If there is any such $n$ with $\phii_{n,s}(n)\downarrow$, for each such $n$, follow the below procedure:

Let $n^*$ be the number such that
\begin{itemize}
\item $n^*\notin\dom(f_s)$ and
\item $n^*\in \hat{f_s}(n)$
\end{itemize}
For each (applicable) $i$, let $n_i$ be such that $f_s^{i+1}(n_i)=n^*$---e.g.\ if $f_s^2(5)=n^*$, then $n_1=5$. Say $n_i$'s are defined for $i\leq I$, and say we have defined $a_0,\dots,a_M$ for some $M\geq m+1$ so far. Set $a_{M+1}=n^*$ and $a_{M+i+2}=n_i$ for $i\leq I$.

If the above procedure is completed for each such $n$ as above, and if we have defined $a_0,\dots,a_k$ for some $k$ by now, we set $A_{s+1}=\{a_0,\dots,a_k\}$. And define $f_{s+1}$ in the obvious way: $f_{s+1}=f_s\cup\{(a_{m+1},a_m),\dots,(a_k,a_{k-1})\}$.
\vspace{0.3cm}

\emph{Verification.} Clearly, the constructed $f=\bigcup_s f_s$ is a total computable function and $A=\bigcup_s A_s$ is an infinite set. To see that $f$ regresses $0'$, note that it is clear from the construction that $f$ regresses $A$: we know $f(a_1)=a_0,f(a_2)=a_1$, and so on. Hence, it suffices to show that $A=0'$. Again by the construction, $A\subseteq 0'$ is clear. And to see $0'\subseteq A$, let $n\in 0'$. Let $s$ be a large enough stage so that $\phii_{n,s}(n)\downarrow$ and $n\in \dom(f_s)$. Then it must be that $n\in A_{s}$ due to the procedure in \emph{Step 2.}

Lastly, we claim that $A$ cannot have a special total regressing function. For if it did, then $A=0'$ would be computable, which is a contradiction. Indeed, if $g$ is a special total computable function which regresses $0'$, then a number $n$ is not in $0'$ if and only if there exists $m\neq n$ such that $\phii_m(m)\downarrow$ and $g^k(m)=g^k(n)$ for some $k$.
\end{proof}
On the next propositions, we show that under certain conditions, co-immune regressive/retraceable sets are majorreducible. On the other hand, we see later that for introenumerable sets, rather than co-immunity, co-c.e.\ is what we need, to get that it is majorreducible.
\begin{prop}\label{retraceable implies umt}
Let $A=\{a_0,a_1,\dots\}$ be a co-immune retraceable set which has a total retracing function. Then $A$ is uniformly majorreducible.
\end{prop}
\begin{proof}
Let $f$ be a special total computable function witnessing the retraceability of $A$, which is possible by Proposition \ref{regressingtree} and Proposition \ref{totaltree}. Let $G$ be a function majorizing $p_A$, and assume we have computed $a_0,\dots,a_n$. Note that we know $a_{n+1}\leq G(n+1)$. Hence, for each $b$ with $a_n<b\leq G(n+1)$, either:

\begin{enumerate}
\item $b=a_{n+1}$ or;

\item $b=a_{j}$ for $j>n+1$ or;

\item $b\notin A$.
\end{enumerate}
And moreover, we can computably decide whether each $b$ is in case $(1)$, or is in case $(2)$/$(3)$---in particular, we can decide if each such $b$ is in case $(1)$ or not. If $f(b)\neq a_n$, then $b$ is in case $(2)$ or $(3)$. If $f(b)=a_n$, then $b$ is either in case $(1)$ or is in case $(3)$. When $f(b)=a_n$, $b$ is in case $(3)$ precisely when there are only finitely many nodes extending $(a_0,\dots,a_n,b)\in T_{f,A}$ on the tree by co-immunity. So we consider the following procedure:

Say $x\in\omega$ is a \emph{candidate for $a_m$} for some $m>n$ if $f(x)$ is a candidate for $a_{m-1}$; we set $a_0,\dots,a_n$ to be the candidates for themselves, as we know what those are by the assumption. If $x$ is a candidate for $a_m$ and there is no number $y$ in $(x,G(m+1)]$ such that $f(y)=x$, then $x$ is a \emph{bad} candidate, in the sense that $x$ must not be in $A$ in this case. In the same sense, if $\sigma\in T_{f,A}$ is such that $\sigma(|\sigma|-1)=x$ is a candidate for $a_m$, we say $x$ is a bad candidate when the following holds: for all $b\in\omega$, if $\sigma^\frown b\in T_{f,A}$, then $b$ is a bad candidate. This is because, in that case, we must have $x\notin A$. Let $b_0,\dots,b_N$ be a list of all candidates for $a_{n+1}$ in $(a_n,G(n+1)]$. If $b_i$ is not in $A$, then there are only finitely many nodes extending $(a_0,\dots,a_n,b_i)$ on the tree $T_{f,A}$. Hence, by the compactness, it must turn out that $b_i$ is a bad candidate for $a_{n+1}$. So we can enumerate the candidates for $a_m$'s for $m>n$ until precisely $N$-many candidates of $b_0,\dots,b_N$ turn out to be the bad candidates, which must happen at some point by the compactness. Let $b_j$ be the remaining candidate. Then $a_{n+1}=b_j$.

With the oracle $G$, this shows that for each $b\in (a_n,G(n+1)]$, we can computably decide whether $b$ is in case $(1)$ or not, and compute $a_{n+1}$. Since $A$ is retraceable, this shows that $A$ is computable from $G$.
\end{proof}
We can weaken the hypothesis of Proposition \ref{retraceable implies umt} to regressive, without weakening the conclusion.
\begin{prop}\label{coimmune regressive implies umt}
If $A=\{a_0,a_1,\dots\}$ is a co-immune regressive set which has a total regressing function, then $A$ is uniformly majorreducible.
\end{prop}
\begin{proof}
Let $f$ be a special total computable function regressing $A$. Let $G$ be a function majorizing $p_A$. With a similar proof to that of Proposition \ref{retraceable implies umt}, we can show that $a_i$'s are computable using $G$ as an oracle. The only difference is that instead of considering the intervals $(a_n,G(n+1)]$'s, we have to consider the intervals $[0,G(n+1))$'s, as regressing functions do not necessarily decrease on $A$. Using a similar argument as in Proposition \ref{retraceable implies umt} and the fact that $A\cap [0,G(n+1)]$ contains an element distinct from $a_0,\dots,a_n$, we can compute $a_i$'s using $G$ as an oracle.

To see that $A$ is computable from $G$, let $b\in\omega$. Note that $b$ is in $A$ if and only if there is some $m\in\omega$ such that $f^m(b)=a_0$ and $b=a_m$. And since $f$ is a special total computable function, for any $b\in\omega$, there exists $m\in\omega$ such that $f^m(b)=a_0$. So it follows that $G$ can compute $A$.
\end{proof}
\begin{cor}
If $A$ is a co-immune regressive/retraceable set which has a total regressing/retracing function, then $A$ is either hyperimmune or computable.
\end{cor}
\begin{proof}
By Proposition \ref{coimmune regressive implies umt}, such a set $A$ is uniformly majorreducible. If $A$ is nonhyperimmune, then there is a total computable function $\phii$ majorizing $A$. The graph of $\phii$ is also computable, thus $A$ is computable by the majorreducibility.
\end{proof}
Note that if $A$ is a co-immune regressive set which is sparse enough, then $A$ is uniformly majorreducible by Proposition \ref{0'comp}: if its principal function grows faster than the $0'$ settling time function, then any function majorizing $p_A$ can compute $0'$, hence $A$. On the other hand, without the assumption of having a total regressing function, a co-immune regressive set may not even be introreducible.
\begin{prop}
There exists a co-immune regressive set which is not introreducible, and hence not majorreducible.
\end{prop}
\begin{proof}
The set we build in Theorem \ref{regressive not intro T} is such a set.
\end{proof}
We end this section with a remark.
\begin{rem}\label{remark}
Jockusch notes that there is a maximal set $M$ such that $\ol{M}$ is uniformly majorreducible \cite[Corollary 6.3]{jock}. Such an $M$ is hyperhypersimple, so $\ol{M}$ has no retraceable subset by Yates \cite[Theorem 6]{yates}. Dekker showed that every infinite regressive set has an infinite retraceable subset \cite{dekker2}. So it follows that $\ol{M}$ has no regressive subset. This also shows that there exists a uniformly majorreducible set which is not regressive. Moreover, this result shows that there exist (uniformly) introreducible sets without regressive subsets.
\end{rem}
\begin{thrm}[Jockusch]
There exists a uniformly majorreducible set without a regressive subset.
\end{thrm}
\section{Introenumerable and introreducible sets}
We show in this section using a similar idea to the proof of Proposition \ref{retraceable implies umt} that every co-c.e.\ introenumerable set is majorreducible. Then we introduce a new method of building an introenumerable or introreducible set and discuss some related open questions. Notably, introenumerability and uniform introenumerability are distinct notions. If a set $A$ is introenumerable, then given some $C\in[A]^\omega$, all we can say is that $A\leq_e C$ via \emph{some} enumeration operator. In this sense, constructing an introenumerable set is a difficult problem.

When $A$ is introenumerable, so that $A\leq_e^i A$, using a failed forcing method, we can show that there exists $C\in[A]^\omega$ with $A\leq^{ui}_e C$ \cite[Proposition 3.7]{green}. In a sense, Corollary \ref{forcingtool} asserts that this can be reversed: we can patch together some sets $B_0,\dots,B_n\subseteq A$ with each $B_i$ being $\geq^{ui}_e A$ and get an introenumerable set. This is a surprising fact because even if $B_0,\dots,B_n$ were all $\geq^{ui}_e A$ via a fixed enumeration operator $\Theta$, it is not necessary that $\Theta$ witnesses $\bigcup_{i\leq n}B_i\geq_e A$. Further, it is not obvious why the set $\bigcup_{i\leq n}B_i$ is even enumeration above $A$. This method gives a rise to building an introenumerable set with some desired properties such as non-uniform introenumerability.

\subsection{Co-c.e.\ introenumerable sets}
In the following proofs and later proofs, we view a string as a set in the following sense. For each string $\sigma\in 2^{<\omega}$, we can view $\sigma$ as the set $\{n:\sigma(n)=1\}$. Hence, $\sigma\cap S=\emptyset$ for some set $S$ means that for all $n$, $\sigma(n)=1$ implies $n\notin S$. Likewise, $\sigma\subseteq S$ means that for all $n$, $\sigma(n)=1$ implies $n\in S$. For a string $\sigma$ and $n\in\omega$, $\sigma\cup\{n\}$ denotes the string $\tau$ such that has the length $\max\{|\sigma|,n+1\}$ and $\tau(k)=1$ if and only if $k=n$ or $\sigma(k)=1$. For two strings $\sigma$ and $\tau$, $\sigma\cup\tau$ is the string $\delta$ of the length $\max\{|\sigma|,|\tau|\}$ such that $\delta(n)=1$ if and only if either $\sigma(n)=1$ or $\tau(n)=1$. Note that for a tree $T\subseteq 2^{<\omega}$, $[T]$ denotes the set of infinite branches of $T$. In the above sense, each element $x\in[T]$ can be viewed as a set, which is the set that has its characteristic function equal to $x$. 
\begin{lem}\label{tree}
Let $A\leq^{i}_e B$ be infinite sets. If there exists a computable tree $T\subseteq 2^{<\omega}$ such that $[T]\neq\emptyset$ and $[T]\subseteq[B]^\omega$, then there is a computable subtree $S\subseteq T$ such that $[S]\neq\emptyset$ and each $X\in[S]$ is uniformly enumeration above $A$, i.e.\ there is an enumeration operator $\Gamma$ such that $\Gamma^X=A$ for all $X\in [S]$.
\end{lem}
\begin{proof}
Let $\Gamma_i$'s be a list of all enumeration operators. Start with the tree $T_0=T$. Having defied $T_n$, we have three cases to consider.

\emph{Case 1.} There exists $\tau\in T_n$ such that there is an infinite branch extending $\tau$ on $T_n$ and for some $m\notin A$, $m\in\Gamma_{n}^\tau$.

\noindent In this case, let $T_{n+1}$ be the subtree of $T_n$ consisting only of those nodes compatible to $\tau$. Proceed constructing $T_{n+2}$.

\emph{Case 2.} There exists $m\in A$ such that for some $X\in[T_n]$, $m\notin \Gamma_{n}^X$.

\noindent In this case, let $T_{n+1}:=\{\sigma\in T_n: m\notin \Gamma_{n,|\sigma|}^\sigma\}$. Note that for each $X\in[T_{n+1}]$, $m\notin \Gamma_n^{X}$. Proceed constructing $T_{n+2}$.

\emph{Case 3.} Neither Case 1 nor Case 2 happens.

\noindent In this case, we just set $T_{n+1}=T_n$ and stop the construction.

Note that each $T_n$, if defined, is a computable tree with an infinite branch. And the construction must stop for some $n$, as otherwise, $\bigcap_n[T_n]$ is nonempty, and any $X\in \bigcap_n[T_n]$ is an infinite subset of $B$ such that $X\ngeq_e A$. Say the construction stops after defining $T_n$. Then $T_n$ is a computable tree such that $[T_n]$ is nonempty and each $X\in[T_n]$ is such that $\Gamma_n^X=A$.
\end{proof}
In \cite[Corollary 6.5]{jock}, Jockusch shows that every co-c.e.\ uniformly introenumerable set is actually uniformly majorreducible. Noting that every nonhyperimmune majorreducible set is computable, he concludes that every co-c.e.\ uniformly introenumerable set is either computable or hyperimmune. The following results get rid of the uniformity assumption.
 \begin{thrm}\label{coceintroe}
 Let $A$ be a co-c.e.\ introenumerable set. Then $A$ is majorreducible.
 \end{thrm}
 \begin{proof}
 Let $\ol{A}=W_e$. Given a function $f$ majorizing $p_A$, to show that $f$ computes $A$, it suffices to show that $A$ is $f$-c.e., as $\ol{A}$ is c.e. Consider the following $f$-computable tree $T$ where its nodes are those $\sigma$'s such that
 
 \begin{enumerate}
 \item $\sigma\cap [0,f(n)]$ has at least $(n+1)$-many elements for each $n$ such that $f(n)<\max(\sigma)$, and
 \item for all $s$, $|\sigma|=s$ implies $\sigma\cap W_{e,s}=\emptyset$.
 \end{enumerate}
 Note that $[T]$ is not empty because $A$, as an infinite binary string, is an infinite branch of $T$. If $X\in [T]$, then $X\subseteq A$ due to the condition (2) above. So by a relativized version of Lemma \ref{tree}, we may let $S\subseteq T$ be an $f$-computable subtree such that $[S]\neq\emptyset$ and each $X\in[S]$ is uniformly $\geq_e A$. Then by \cite[Proposition 3.6]{andrews}, $A$ is $f$-c.e. Indeed, a number $n$ is in $A$ if and only if $\{X\in 2^\omega: n\in \Gamma^X\text{ or }X\notin [S]\}=2^\omega$. And by the compactness, the latter condition is $\Sigma_1^0$ in $f$.
 \end{proof}
\begin{cor}\label{intro-eHI}
Let $A$ be a co-c.e.\ introenumerable set. Then $A$ is either computable or hyperimmune.
\end{cor}
\begin{proof}
This is immediate from Theorem \ref{coceintroe} noting that every majorreducible set is either computable or hyperimmune.
\end{proof}
From Corollary \ref{intro-eHI}, it is tempting to believe the stronger result that every co-c.e.\ introenumerable set is actually uniformly majorreducible. However, this should not be true as Lachlan (in private conversation \cite{jock}, where the proof can be found in \cite[Theorem 4.1]{jock}) constructs a co-c.e.\ introreducible set which is not uniformly introreducible. We will refer to this result as Lachlan's Construction.

Certainly, there exists a noncomputable co-c.e.\ introenumerable set. One simple example is a noncomputable co-c.e.\ retraceable set. It is Dekker and Myhill's result in \cite[Theorem T3]{dekker} that every c.e.\ degree contains a co-c.e.\ retraceable set.
 \begin{cor}\label{introeimpliesintrot}
 If $A$ is a co-c.e.\ introenumerable set, then $A$ is introreducible.
 \end{cor}
 \begin{cor}\label{majortnotuniform}
 There exists a majorreducible set which is not uniformly majorreducible.
 \end{cor}
 \begin{proof}
 Let $A$ be a co-c.e.\ introenumerable set which is not uniformly introenumerable; such a set certainly exists, e.g.\ the Lachlan's Construction. So by Theorem \ref{coceintroe}, $A$ is majorreducible. For if $A$ were also uniformly majorreducible, then $A$ would also be uniformly introreducible, and hence uniformly introenumerable. So $A$ is not uniformly majorreducible.
 \end{proof}
 \begin{cor}
 There exists a majorenumerable set which is not uniformly majorenumerable.
 \end{cor}
 \begin{proof}
 The set considered in the proof of Corollary \ref{majortnotuniform} works.
 \end{proof}
 \subsection{Question arising from Jockusch's work}
 Now we address one question that arises from the Jockusch's work. In \cite{jock}, he shows that any finite union of uniformly introreducible sets of the same Turing degree is introreducible---it is a corollary to \cite[Theorem 3.6]{jock}, one of the main results in the paper:
  \begin{thrm}[Jockusch's Theorem]\label{jockmain}
 Assume $A_0,\dots,A_n$ are infinite sets such that $A_1$ is uniformly introenumerable and $A_1,\dots,A_n$ are uniformly introreducible. For each $C\subseteq\bigcup_{i=0}^nA_i$, $A_0$ is c.e.\ in $C$ provided that $C$ has infinitely many elements from $A_0\setminus (A_1\cup\cdots\cup A_n)$. In particular, by induction, any infinite $C\subseteq \bigcup_{i=0}^n A_i$ can enumerate some $A_i$.
 \end{thrm}
 We will refer to this theorem as Jockusch's Theorem.
 \begin{cor}[Jockusch]\label{jockcor}
If $A_0,\dots, A_n$ are uniformly introreducible sets of the same Turing degree, then the union $\bigcup_{i=0}^nA_i$ is introreducible. Moreover, the union also has the same Turing degree.
 \end{cor}
 \begin{proof}
 Let $A$ be the union $\bigcup_{i=0}^nA_i$ and consider an infinite subset $C\in[A]^\omega$. By Jockusch's Theorem, $C$ can enumerate some $A_i$, and hence compute an infinite subset $C_0\in[A_i]^\omega$. By the introreducibility of $A_i$, this implies that $C$ can compute $A_i$, which can compute $A_0,\dots,A_n$, and hence $A$. To see that $A \equiv_T A_0$, $A\leq_T A_0$ is clear. And we just showed that any infinite subset of $A$ can compute some $A_i$, which has the same Turing degree with $A_0$.
 \end{proof} 
 The question is whether the converse is true or not.
 \begin{quest}\ 
 \begin{enumerate}
 \item Is every introreducible set a union of finitely many uniformly introreducible sets?
 \item Can you further assume that all of those sets have the same Turing degree?
 \end{enumerate}
\end{quest}
We make a relevant observation showing that the second item comes for free if the first item is true, i.e., it is that either every introreducible set is a union of finitely many uniformly introreducible sets of the same Turing degree or there is an introreducible set which cannot be a union of finitely many uniformly introreducible sets. Consider an infinite set $A$, and assume that $A$ is a union of finitely many sets $A_0,\dots,A_n$, all of which are uniformly introreducible. Certainly, if the union $\bigcup_{i=09}^n A_i$ is a nontrivial union, i.e.\ each $A_i\setminus(\bigcup_{j\neq i}A_j)$ is infinite, then each $A_i$ is c.e.\ in $A$ by Jockusch's Theorem, so that $A$ can compute an infinite subset of each $A_i$. Then by the introreducibility of $A_i$, it follows that $A_i\leq_T A$ for each $i$. If we further assume that $A$ is introreducible, then $A\leq_T A_i$ as well, so that we have $A_i\equiv_T A$ for each $i$. Hence, $A$ is actually a union of the finitely many uniformly introreducible sets, all of those sets having the same Turing degree.

Now, assume that our infinite set $A$ is the union $\bigcup_{i=0}^n A_i$ where $A_i$'s are uniformly introreducible, but the union is not necessarily a nontrivial union. We claim that we can find uniformly introtreducible sets $D_0,\dots,D_m\subseteq A$ with $A=\bigcup_{i\leq m}D_i$ being a nontrivial union and $A\equiv_T D_i$ for $i\leq m$. If $A_0\setminus (\bigcup_{i\neq 0,i\leq n}A_i)$ is a finite set, then add those finitely many elements of $A_0\setminus (\bigcup_{i\neq 0,i\leq n}A_i)$ to $A_1$ and disregard $A_0$. Then reindex the sets by $A_0=A_1,A_1=A_2,\dots,A_{n-1}=A_n$. Otherwise, if $A_0\setminus (\bigcup_{i\neq 0,i\leq n}A_i)$ is an infinite set, then let $C_0:=A_0$. Having defined $C_k$ and having $C_0,\dots,C_k,A_{k+1},\dots,A_l$ for some $l$, check whether $A_{k+1}\setminus [(\bigcup_{i\leq k}C_i)\cup(\bigcup_{k+1<i\leq l}A_i)]$ is finite or not. If finite, then add those finitely many elements to $A_{k+2}$ and reindex the remaining $A_i$'s as $A_{k+1}=A_{k+2},\dots,A_{l-1}=A_l$. If infinite, then let $C_{k+1}=A_{k+1}$. Note that it is possible to have only $C_0$ defined at the end: for instance, we could have $A_0=\omega$ and $A_1=\omega\setminus \{0\}$ where $A=\omega$. In this case, we would disregard $A_0$, add $0$ to $A_1$, reindex $A_1\cup\{0\}$ as $A_0$, then let $C_0=A_0=\omega$. Also note that it is possible to have some elements of $A$ not in any of the $C_i$'s at the end: we could have $A_0=\{\text{evens}\}\setminus \{0\}$, $A_1=\{\text{odds}\}$, and $A_2=\{0\}\cup\{\text{odds}\}\setminus \{4n+1:n\in\omega\}$ where $A=\omega$. In this case, we would have at the end $C_0=A_0$, $C_1=A_1$, and $C_2$ is undefined, so that $0$ is not in $C_0\cup C_1$.
 
 Hence, if we proceed in the above way, then we end up with $C_0,\dots,C_m$ for some $m$ where $A\setminus (\bigcup_{i\leq m}C_i)$ is a finite set. Noting that a finite adjustment of a uniformly introreducible set is still uniformly introreducible, we may let $A=\bigcup_{i\leq N}D_i$ for some $N$ and $D_i$'s such that each $D_i\setminus(\bigcup_{j\neq i,j\leq N}D_j)$ is infinite and each $D_i$ is uniformly introreducible---e.g., we let $D_0=C_0,\dots,D_{m-1}=C_{m-1}$ and $D_m=C_m\cup (A\setminus (\bigcup_{i\leq m}C_i))$. Then apply Jockusch's Theorem to $A$ and each $D_i$ (in place of $C$ and $A_1$) to get that each $D_i$ is c.e.\ in $A$. By the introreducibility of $D_i$'s, this implies that $A$ computes each $D_i$. If we further assume that $A$ is introreducible, then $A\leq_T D_i$ for each $i$ as well.
 
 Hence, combining this observation with Corollary \ref{jockcor}, we get:
 \begin{cor}\label{finiteunion}
 Let $A=A_0\cup\cdots\cup A_n$, where each $A_i$ is an infinite set. Assume all $A_i$'s are uniformly introreducible. Then $A$ is introreducible if and only if there exists uniformly introreducible sets $D_0,\dots, D_m\subseteq A$ for some $m$ with $A=\bigcup_{i=0}^mD_i$ and $A\equiv_T D_i$ for each $i$. Note that for the `if' direction, assuming $D_0\equiv_T D_i$ for each $i$ suffices in place of $A\equiv_T D_i$ for each $i$ by Corollary \ref{jockcor}.
 \end{cor}
 Thus, it is natural to ask whether every introreducible set is a union of finitely many uniformly introreducible sets, all of them having the same Turing degree. A very strong assertion, that every introreducible set is uniformly introreducible, is shown to be false by Lachlan's Construction. An interesting point to note is that the set Lachlan constructs, although not uniformly introreducible, is still a union of two uniformly introreducible sets; in fact a disjoint union.
 
 \subsection{Building an introenumerable set}There are many more interesting related questions arising from the following theorem; we note that the core idea of the proof comes from the proof of Jockusch's Theorem.
 \begin{thrm}\label{decomposition}
 If $A,B_0,\dots,B_n$ are infinite sets such that $A\leq_{e}^{ui} B_i$ for each $i\leq n$, then $A\leq_e^i \bigcup_{i=0}^nB_i$.
 \end{thrm}
 \begin{proof}
 The proof is done by induction on $n$---the base case $n=0$ is trivial. Assume the theorem works for $n-1$, and we prove it for $n$. Let $\Theta_i$ be an enumeration operator witnessing $A\leq_e^{ui}B_i$ for each $i\leq n$. Let $C\subseteq \bigcup_{i=0}^nB_i$ be an infinite subset and assume that $A$ is not $C$-c.e.\ for a contradiction. Consider the theory $T$ in an appropriate relational language where the axioms of $T$ say
\begin{itemize}
 \item $E\subseteq B_i\rimp A(k)=1$ (for each $E,k$ such that $k\in \Theta_i(E)$);
 \item $B_0(k)=1\vee\cdots\vee B_n(k)=1$ (for each $k\in C$).
\end{itemize}
Clearly, $A,B_0,\dots,B_n$ in the assumption with the natural interpretation model $T$, so that $\{n:T\proves A(n)=1\}\subseteq A$. And by the assumption, since $T$ is a $C$-c.e.\ theory, there exists $m\in A$ such that $T\nproves A(m)=1$. Say a formula $\phii$ is \emph{special} if it is a true conjunction of statements of the form `$l\in B_i$' for $l\in C$ such that $T\nproves (\phii\rimp A(m)=1)$; e.g., if $1,3\in C$, $1\in B_0, 3\in B_2$, and $T\nproves (B_0(1)=1\wedge B_2(3)=1\rimp A(m)=1)$, then $\phii:=(B_0(1)=1\wedge B_2(3)=1)$ is a special formula. For formulas $\phii$ and $\psi$, we say $\psi$ is a proper special extension of $\phii$ if both $\phii$ and $\psi$ are special formulas, and the statements occurring in $\psi$ form a proper super set of the set of those occurring in $\phii$. Note that our assumption tells that the empty formula is special.

Note that we cannot have an infinite sequence of properly extending special formulas. For if $\phii_0,\phii_1,\dots$ are properly extending special formulas, then there is some $j\leq n$ such that the statements occurring in $\phii_i$'s give an infinite subset of $B_j$. Since $A\leq^{ui}_e B_j$ via $\Theta_j$, there must exist some number $N$ such that $T\cup\{\phii_N\}\proves A(m)=1$. Hence, we may let $\phii$ be a special formula which has no proper special extension.

For each $k,j\in\omega$, let $\phii_{k,j}$ be the formula `$(B_k(j)=1\wedge\phii)\rimp A(m)=1$'. If $j\in B_k\cap C$ and `$B_k(j)=1$' is not occurring in $\phii$, then $T\proves\phii_{k,j}$, as otherwise, $\phii_{k,j}$ would be a proper special extension of $\phii$. For each fixed $j\in C$, let $p(j)$ be the number of $k$'s such that $T\proves \phii_{k,j}$. So $p(j)\geq 1$ for all but finitely many $j$ in $C$. Moreover, $p(j)<n$ for each $j\in C$. For if $j\in C$ and $p(j)=n$, then since `$B_0(j)=1\vee\cdots\vee B_n(j)=1$' is an axiom of $T$, $T\proves(\phii_{0,j}\wedge\cdots\wedge \phii_{n,j})$ implies that $T\proves(\phii\rimp A(m)=1)$, which contradicts that $\phii$ is special. Let $J\in\omega$ be the largest number such that there are infinitely many $j\in C$ with $p(j)=J$. Let $\CB$ be a collection of $J$-many $B_k$'s such that there are infinitely many $j\in C$ such that $T\proves\phii_{k,j}$ for each $k$ with $B_k\in\CB$. Let $s$ be large enough so that $p(j)\leq J$ for all $j\in C$ with $j\geq s$ and that $s$ is larger than all statements occurring in $\phii$. Define
$$S=\{j:j\geq s\hs\&\hs j\in C\hs\&\hs T\proves\phii_{k,j}\text{ for each }k\text{ with }B_k\in\CB\}.$$
Clearly, $S$ is an infinite subset of $\bigcup\CB$: if $j\in S$ and $j\notin \bigcup\CB$, then let $j\in B_l$ where $B_l\notin\CB$. Then it follows that $T\proves \phii_{l,j}$, contradicting the choice of $s$.

Note that $S$ is a $C$-c.e.\ set. Hence, we may let $S_0\subseteq S$ be an infinite $C$-computable subset. By the induction hypothesis, as $\CB$ contains $<n$ many sets, $A$ is $S_0$-c.e.\ and thus $A$ is $C$-c.e.\ contradicting our assumption.
 \end{proof}
 \begin{cor}\label{4.12}
 If $A,B_0,\dots,B_n$ are infinite sets such that $A\leq_T^{ui}B_i$ for each $i\leq n$, then $A\leq_T^i\bigcup_{i=0}^nB_i$.
 \end{cor}
 \begin{proof}
 Recall that $\leq^{ui}_{c.e.}$ and $\leq^{ui}_e$ are equivalent. Also, note that $X\leq_T^{ui}Y$ if and only if $X\leq_{c.e.}^{ui} Y$ and $\ol{X}\leq_{c.e.}^{ui}Y$, and $X\leq_T^{i}Y$ if and only if $X\leq_{c.e.}^{i} Y$ and $\ol{X}\leq_{c.e.}^{i}Y$. So the corollary follows directly from Theorem \ref{decomposition} and the two facts just noted.
 \end{proof}
\begin{cor}
 Assume $B=B_0\cup\cdots\cup B_n$ where $B_0,\dots,B_n$ are uniformly introenumerable sets. If $B_i$'s have the same enumeration degree, then $B$ is introenumerable.
 \end{cor}
 \begin{proof}
 In this case, we have $B_j\leq^{ui}_e B_i$ for all $i,j\leq n$, and hence $\bigcup_{j=0}^nB_j\leq^{ui}_e B_i$ for all $i\leq n$. By Theorem \ref{decomposition}, we have $B\leq^i_e B$.
 \end{proof}
  The following takes $A$ in Theorem \ref{decomposition} equal to the union $\bigcup_{i\leq n}B_i$. This gives us a new method of building a nontrivial introenumerable or introreducible set.
 \begin{cor}\label{forcingtool}
 If $A,B_0,\dots,B_n$ are infinite sets such that $A=\bigcup_{i\leq n}B_i$ and $A\leq^{ui}_e B_i$ for each $i$, then $A$ is introenumerable. Likewise, if $A\leq_T^{ui}B_i$ for each $i$, then $A$ is introreducible.
 \end{cor}
 In fact, the following theorem shows that an introenumerable set gotten by Corollary \ref{forcingtool} is not necessarily uniformly introenumerable.
 \begin{thrm}\label{ienotuie forcing proof}
 There exists an introenumerable set which is not uniformly introenumerable, not introreducible, and not majorenumerable.
 \end{thrm}
 Note that such an introenumerable set is as nontrivial as possible in the sense that it avoids all notions which it can avoid. We use a modification of the forcing notion introduced in \cite[Theorem 1.3(b)]{green}.
 \begin{proof}
For each $i\in\omega$, let $X_i$ be the set of numbers divisible by the $(2i)$th prime number, and let $Y_i$ be the set of numbers divisible by the $(2i+1)$th prime number. Let $\Theta_0$ be the enumeration operator such that for a finite set $E\subseteq\omega$, $n\in \Theta_0(E)$ if and only if $E\cap X_n\neq\emptyset$. Likewise, let $\Theta_1$ be the enumeration operator such that $n\in\Theta_1(E)$ if and only if $E\cap Y_n\neq\emptyset$. We will build infinite sets $G_0$ and $G_1$ such that $G:=G_0\cup G_1$ is $\leq^{ui}_e G_0$ via $\Theta_0$ and $G\leq^{ui}_e G_1$ via $\Theta_1$. Note that for each finite $E\subseteq \omega$, $\bigcap_{m\in E}X_m\cap\bigcap_{m\notin E}\ol{X_m}$ is an infinite set, and likewise for $Y_m$'s. Let $A_E:=\bigcap_{m\in E}X_m\cap\bigcap_{m\notin E}\ol{X_m}$ and $B_E:=\bigcap_{m\in E}Y_m\cap\bigcap_{m\notin E}\ol{Y_m}$. Note that $A_{E_1}\cap B_{E_2}$ is infinite for finite sets $E_1,E_2\subseteq\omega$ as well.

The forcing notion $\p$ consists of the triples $(\sigma,\tau,k)$ such that $\sigma,\tau\in 2^{<\omega}$ and $k\in\omega$ satisfy the following:
\begin{itemize}
\item If $(\sigma\cup\tau)(n)=0$, then $X_n\cap \sigma=\emptyset$ and $Y_n\cap \tau=\emptyset$.
\item $k<|\sigma|,|\tau|$. 
\end{itemize}
For $(\sigma_0,\tau_0,k_0),(\sigma_1,\tau_1,k_1)\in\p$, we say $(\sigma_1,\tau_1,k_1)\leq_\p(\sigma_0,\tau_0,k_0)$ if
\begin{itemize}
\item $\sigma_0\preceq\sigma_1,\tau_0\preceq\tau_1,k_0\leq k_1$.
\item $\forall n<k_0$, $(\sigma_0\cup\tau_0)(n)=1$ implies $(\sigma_1\setminus\sigma_0)\subseteq X_n$ and $(\tau_1\setminus \tau_0)\subseteq Y_n$.
\end{itemize}
For a condition $p\in\p$, we write $(\sigma_p,\tau_p,k_p)$ for the corresponding triple. Let $E_p\subseteq\omega$ denote the finite set $\{n\in\omega:(\sigma_p\cup\tau_p)(n)=1\}$. Clearly, for each $p\in\p$, we can find a nontrivial extension $q\leq_\p p$ in $\p$. For instance, we can let $n_0\in A_{E_p}$ be larger than $|\sigma_p|$ and $n_1\in B_{E_p}$ be larger than $|\tau_p|$. Then $(\sigma_p\cup\{n_0\},\tau_p\cup\{n_1\},k_p+1)$ extends $p$ in $\p$. Let $\G$ be a sufficiently generic filter on $\p$ and let $G_0=\bigcup\{\sigma_p:p\in\G\}$, $G_1=\bigcup\{\tau_p:p\in\G\}$. Let $G=G_0\cup G_1$. Then $G_0$ and $G_1$ are forced to be infinite sets such that $G\leq_e^{ui}G_0$ via $\Theta_0$ and $G\leq_e^{ui}G_1$ via $\Theta_1$, so that $G$ is introenumerable by Corollary \ref{forcingtool}.

We claim that $G$ is not uniformly introenumerable by the sufficient genericity of $\G$. Assume otherwise, so that $G$ is uniformly introenumerable via an enumeration operator $\Gamma$. We show that the set
$$D:=\{q\in\p:\exists n\hs n\notin E_q\hs\&\hs n\in\Gamma(E_q)\}$$
is dense above $\G$ for a contradiction. Fix $p\in\G$ and let $n\in G_0\setminus G_1$ be larger than $|\sigma_p|$ and $|\tau_p|$---it is easy to show that the empty condition forces the both sets $G_0\setminus G_1$ and $G_1\setminus G_0$ to be infinite. Note that for each finite $F\subseteq\omega$, the empty condition already forces $G_0\cap B_F$ to be an infinite set: given any $p\in\p$, we may let $m\in A_{E_p}\cap B_F$ be larger than $|\sigma_p|$, so that $(\sigma_p\cup\{m\},\tau_p,k_p)$ is a condition extending $p$. Hence, $Y:=(G_0\cap B_{E_p})\cap (n,\infty)$ is an infinite subset of $G$. Then $n\in \Gamma(Y)$ by the assumption, so we may let $u\subseteq Y$ be a finite subset such that $n\in\Gamma(u)$. Then $q:=(\sigma_p,\tau_p\cup u,k_p)\in\p$ is a condition extending $p$, and note that $n\in \Gamma(E_q)$ while $n\notin E_q$. Thus, we have $q\in D$.

To see that $G$ is not introreducible, we prove the following claim: given $p\in \G$, a Turing operator $\Phi$, and $S\subseteq E_p$, either
\begin{enumerate}
\item there exists $q\leq_\p p$ and $S^*\subseteq E_q$ such that $S^*$ extends $S$ as an end extension and $\Phi^{S^*}$ disagrees with $E_q$, or
\item there exists $n$, $p_0\leq_\p p$, and $S_0\subseteq E_{p_0}$ extending $S$ as an end extension such that for all $p_1\leq_\p p_0$ and $S_1\subseteq E_{p_1}$ extending $S_0$ as an end extension, $\Phi^{S_1}(n)\uparrow$.
\end{enumerate}
 If the claim is true, then by the sufficient genericity, the empty condition forces that there exists an infinite subset $Y\subseteq G$ with $G\nleq_T Y$. Hence, fix $p$, $\Phi$, and $S\subseteq E_p$. Assume that $(2)$ is not true with $p$, $\Phi$, and $S$. Let $n\in A_{E_p}$ be a number larger than both $|\sigma_p|$ and $|\tau_p|$. Consider the extension $p_0$ of $p$ which is gotten by concatenating $0$'s to $\sigma_p$ and $\tau_p$ up to $n$. As $(2)$ is not true, we may let $p_1\leq_\p p_0$ and $S_1\subseteq E_{p_1}$ be so that $S_1$ extends $S$ as an end extension and $\Phi^{S_1}(n)\downarrow$. If $\Phi^{S_1}(n)=1$, then $q:=p_1$ is a desired extension. If $\Phi^{S_1}(n)=0$, then $q:=(\sigma_{p_1}\cup\{n\},\tau_{p_1},k_p)$ is a desired extension.

 Lastly, the sufficient genericity ensures that $G$ is not majorenumerable. Indeed, let $(H_i,i\in\omega)$ be a list of all hyperarithmetic sets. Let $p\in\G$ be a condition and let $i\in \omega$. If $H_i$ is an infinite set, pick $n\in H_i$ larger than $|\sigma_p|+|\tau_p|$. We consider $q\leq_\p p$ gotten by concatenating $0$'s to $\sigma_p$ and $\tau_p$ up to $n$. Clearly, $q\forces H_i\nsubseteq G$. Hence, the sufficient genericity ensures that $G$ has no infinite hyperarithmetic subset. It follows from Corollary \ref{jock sol} and Corollary \ref{refinable} that $G$ is not majorenumerable.
\end{proof}
\begin{deft}
An introenumerable set $A$ is \emph{finitely decomposable} if there exists $A_0,\dots,A_n$ such that $A=A_0\cup\cdots\cup A_n$ and $A\leq^{ui}_e A_i$ for each $i\leq n$. An introreducible set $A$ is \emph{finitely decomposable} if there exists $A_0,\dots,A_n$ such that $A=A_0\cup\cdots \cup A_n$ and $A\leq^{ui}_T A_i$ for each $i$.
\end{deft}
 In \cite[Theorem 1.5]{green}, they show that if $A$ is introenumerable, then there is $B\in[A]^\omega$ such that $A\leq^{ui}_T B$. If $\hat{B}$ is a finite modification of $B$, then $A\leq^{ui}_T \hat{B}$ still holds. Hence, every introenumerable (or introreducible) set is a countable union of sets $B_0,B_1,\dots$ such that $A\leq^{ui}_T B_i$ for all $i\in\omega$. If such a union can be chosen to be a finite union, then the union is introreducible by Corollary \ref{4.12}. Hence, any introenumerable set which is not introreducible cannot be a union of finitely many such sets. On the other hand, it is still open whether every introenumerable set $A$ is a union of finitely many sets $A_0,\dots,A_n$ such that $A\leq_e^{ui}A_i$ for $i\leq n$, i.e.\ whether every introenumerable set is finitely decomposable. Thus, we have the following open questions.
 \begin{quest}\label{open questions}\ 
 \begin{enumerate}
 \item Does there exist an introreducible set which is not a union of finitely many uniformly introreducible sets?
 
 \item Does there exist an introreducible set which does not have a uniformly introreducible subset of the same Turing degree? If so, such a set is an example of an introreducible set which cannot be a union of finitely many uniformly introreducible sets.

 \item Does there exist an introenumerable set which does not have a uniformly introenumerable subset of the same enumeration degree? If so, such a set is an example of an introenumerable set which cannot be a union of finitely many uniformly introenumerable sets of the same enumeration degree.

 \item Does there exist an introreducible set which is not finitely decomposable?

 \item Does there exist an introenumerable set which is not finitely decomposable?
 \end{enumerate}
 \end{quest}
Related to item $(2)$ is the question: can you find an introreducible set $A$ which has no infinite subset $B\subseteq A$ such that $B\equiv_T A$ and $A\leq_T^{ui} B$? In fact, this question is equivalent to item $(2)$. We answer positively in the case that `introreducible' is replaced by `introenumerable' and $\equiv_T$ is replaced by $\equiv_e$.
\begin{prop}\label{forcingproof}
There exists a uniformly introenumerable set $G$ such that for any infinite set $A$ with $A\leq_e G$, we have $\neg[A\geq^{ui}_T G]$.
\end{prop}
\begin{proof}
Let $X_i$ and $A_E$'s be as in the proof of Theorem \ref{ienotuie forcing proof}. The forcing notion we consider in this proof is the set $\p$ consisting of the pairs $(\sigma,k)$ such that $\sigma\in 2^{<\omega}$ and $k\in\omega$ satisfy the followings:
\begin{itemize}
\item If $\sigma(n)=0$, then $X_n\cap \sigma=\emptyset$.
\item $k<|\sigma|$.
\end{itemize}
For $(\sigma_0,k_0),(\sigma_1,k_1)\in\p$, we say $(\sigma_1,k_1)\leq_\p(\sigma_0,k_0)$ if
\begin{itemize}
\item $\sigma_0\preceq\sigma_1,k_0\leq k_1$.
\item $\forall n<k_0$, $\sigma_0(n)=1$ implies $(\sigma_1\setminus\sigma_0)\subseteq X_n$.
\end{itemize}
For a condition $p\in\p$, we write $(\sigma_p,k_p)$ for the corresponding pair. Let $E_p\subseteq\omega$ denote the finite set, which is $\sigma_p$ viewed as a set. For each $p\in\p$, there exists a nontrivial extension $q\leq_p p$ in $\p$. For instance, we can pick a number $n_0\in A_{E_p}$ larger than $|\sigma_p|$ and consider $(\sigma_p\cup\{n_0\},k_p+1)\in\p$. Let $\G$ be a sufficiently generic filter on $\p$, and let $G=\bigcup\{\sigma_p:p\in\G\}$. Let $\Theta_0$ be the enumeration operator defined in the proof of Theorem \ref{ienotuie forcing proof}. Then the empty condition already forces $G$ to be an infinite set and $k_p$ to increase to infinity, so that $G$ is uniformly introenumerable via $\Theta_0$.

Let $\Phi$ be a Turing operator, and let $\Gamma$ be an enumeration operator. We claim that the empty condition forces $\neg[\Gamma(G)\geq_T^{ui} G\text{ via }\Phi]$. Assume on the contrary that for all $C\in[\Gamma(G)]^\omega$, $\Phi^C=G$. We show that the set 
$$D:=\{p\in\p:\exists n\hs\exists F\subseteq\Gamma(E_p)\hs \Phi^{F}(n)=0\hs\&\hs n\in E_p\}$$
is dense above $\G$ to get a contradiction---for if $p\in \G\cap D$, then $F\subseteq\Gamma(E_p)$ can be extended to an infinite subset of $\Gamma(G)$ while preserving all bits which witness the computation $\Phi^{F}(n)\downarrow=0$. Fix $p\in \G$. Note that for each finite set $E\subseteq\omega$, the empty condition forces $A_E\setminus G$ to be an infinite set, because for each $p\in\p$, $(\sigma_p^\frown 0,k_p+1)\in \p$ extends $p$. Hence, we may let $m>|\sigma_p|$ be such that $m\notin G$ and $m\in A_{E_p}$. By the assumption, we have $\Phi^{\Gamma(G)}(m)=0$, so we may let $q\in\G$ be such that $q\leq_\p p$ and $\Phi^{\Gamma(E_q)}(m)=0$. Then $r:=(\sigma_q\cup\{m\},k_p)$ is a condition in $\p$ which extends $p$. Moreover, $\Gamma(E_q)\subseteq \Gamma(E_r)$, so $r\in D$.
\end{proof}
If one can prove Proposition \ref{forcingproof} with $\geq^{ui}_T$ replaced by $\geq^{ui}_e$ and `uniformly introenumerable' replaced by `introenumerable', then that shows an existence of an introenumerable set without a uniformly introenumerable set of the same enumeration degree; in particular, it answers item (3). Note that $\leq_e$ cannot be replaced by $\leq_T$ as every Turing degree contains a uniformly introreducible set.

 When $A$ is introenumerable, i.e.\ $A\leq_e^i A$, there is some infinite subset $B\subseteq A$ with $A\leq^{ui}_T B$ \cite[Proposition 5.3]{green}. Informally, by thinning out the right hand side, we can find a subset where $A$'s both positive and negative information are very densely coded in it. What about the other way around? By thinning out the left hand side instead, can we find an infinite subset $B\subseteq A$ with $B\leq^{ui}_T A$? The following shows that the answer is strongly negative.
 \begin{prop}
 There exists a uniformly introenumerable set $G$ such that there is no $C\in[G]^\omega$ with $C\leq^i_T G$.
 \end{prop}
 The proof below is a slight modification of the proof of \cite[Theorem 1.3(b)]{green} which states: there exists $A$ and $B$ with $A\leq^{ui}_e B$ such that there is no $C\in[A]^\omega$ with $C\leq^i_TB$.
 \begin{proof}
 We consider the forcing notion $\hat{\p}$ which is almost same as the forcing notion $\p$ from the proof of Proposition \ref{forcingproof}, but with the following difference: the conditions of $\hat{\p}$ are the triples $(\sigma,\tau,k)$ such that $(\sigma,\tau,k)\in\hat{\p}$ if and only if $(\sigma,k)\in\p$ and $\tau\subseteq\sigma$. And for two conditions $(\sigma_0,\tau_0,k_0)$ and $(\sigma_1,\tau_1,k_1)$ in $\hat{\p}$, we define $\leq_{\hat{\p}}$ in the following way: $(\sigma_1,\tau_1,k_1)\leq_{\hat{\p}}(\sigma_0,\tau_0,k_0)$ if and only if $(\sigma_1,k_1)\leq_\p(\sigma_0,k_0)$ and $\tau_1\succeq\tau_0$. For a condition $p\in\hat{\p}$, let $(\sigma_p,\tau_p,k_p)$ denote the corresponding triple to $p$. Let $\G$ be a sufficiently generic filter on $\hat{\p}$. We let $G=\bigcup\{\sigma_p:p\in\G\}$ and $X=\bigcup\{\tau_p:p\in\G\}$.
 
 As in the proof of Proposition \ref{forcingproof}, $G$ is uniformly introenumerable. And clearly, the empty condition forces $X\subseteq G$ to be an infinite subset. Assume for contradiction that there is some $C\in[G]^\omega$ such that $C\leq^i_T G$. Then in particular, there exist Turing operators $\Phi$ and $\Psi$ such that $\Phi(G)=\Psi(X)=C$.

 Let $p\in\G$ be a condition. Due to the sufficient genericity, we may assume, by considering an extension if necessary, that $p$ forces that $\Phi(G)$ is total and infinite, and that $\Psi(X)$ is total and a subset of $G$. Hence, there exists $i>|\sigma_p|$ such that for some condition $q\leq_{\hat{p}}p$, we have $\Phi^{\sigma_q}(i)\downarrow=1$. Let $j:=i+|\sigma_q|$. Note that $r_0:=(\sigma_p^\frown 0^j,\tau_p^\frown 0^j,k_p)$ is a condition extending $p$ in $\hat{\p}$, where $0^j$ denotes $j$-many $0$'s. So there exists a condition $r\in\hat{\p}$ extending $r_0$ such that $\Psi^{\tau_r}(i)\downarrow=0$ by the assumption that $\Psi(X)$ is forced to be total and a subset of $G$.
 
 Now, we can combine the two conditions $q$ and $r$ to get a contradiction. Let $\sigma^*:=\sigma_q\cup\sigma_r$. First of all, we claim that $p^*:=(\sigma^*,\tau_r,k_p)$ is a condition of $\hat{\p}$. We certainly have $\tau_r\subseteq \sigma^*$. And since $(\sigma_q,k_q)$ and $(\sigma_r,k_r)$ are conditions of $\p$, $(\sigma^*,k_p)$ is a condition of $\p$ as well. Secondly, we check that $p^*$ extends $p$ in $\hat{\p}$. We clearly have $\tau_r\succeq \tau_p$. And note that $(\sigma^*\setminus \sigma_p)=(\sigma_q\setminus \sigma_p)\cup(\sigma_r\setminus \sigma_p)$. So the fact that $q$ and $r$ both extend $p$ implies that $(\sigma^*\setminus \sigma_p)\subseteq X_n$ for each $n<k_p$ with $\sigma_p(n)=1$. Lastly, $\Phi^{\sigma^*}(i)\downarrow=1$ because $\sigma^*\restriction_{|\sigma_q|}=\sigma_q$. So we have found a condition $p^*\leq_{\hat{\p}}p$ such that `$\Phi(G)=\Psi(X)$' does not hold on it. By the sufficient genericity, this is a contradiction.
 \end{proof}
 The next proposition shows that we cannot build an introreducible set that does not have a uniformly introreducible subset using the tool from Corollary \ref{forcingtool}. Hence, if there exists an introreducible set that does not have a uniformly introreducible subset, then item (4) from the Open Questions \ref{open questions} has a positive answer.
\begin{prop}
Let $A$ be an introreducible set. If $A$ is finitely decomposable, then $A$ has a uniformly introreducible subset.
\end{prop}
\begin{proof}
Let $A=\bigcup_{i\leq n}A_i$ be a finite decomposition of $A$. Let $(2^{<\omega})^{n+1}$ denote the product of $(n+1)$-many $2^{<\omega}$'s. Fix an $A$-computable bijection $\pi:(2^{<\omega})^{n+1}\to A$. Let 
$$\hat{A}:=\{\pi(\sigma_0,\dots,\sigma_n):|\sigma_0|=\cdots=|\sigma_n|\hs\&\hs \sigma_i\prec A_i\text{ for each }i\leq n\}.$$
Then for some $i\leq n$, $\hat{A}\cap A_i$ is infinite. Say $\hat{A_i}:=\hat{A}\cap A_i$ is infinite. We show that $\hat{A_i}$ is uniformly introreducible. Let $C$ be an infinite subset of $\hat{A_i}$. Then $C$ is an infinite subset of both $\hat{A}$ and $A_i$. So $C$ can uniformly compute $A$ and $\pi$, so $C$ can uniformly compute $A_0,\dots,A_n$ from infinitely many elements of $\hat{A}$. So $C$ can uniformly compute $\hat{A}$ and $A_i$.
\end{proof}
Certainly, Corollary \ref{forcingtool} does not allow one to build an introenumerable or introreducible set with full genericity if there could exist an introenumerable or introreducible set which is not finitely decomposable. And the above proposition shows that every finitely decomposable introreducible set has a uniformly introreducible subset. We end this section by noting that for an introenumerable set $A$, if one can hyperarithmetically find a uniformity from $A$, then $A$ has a uniformly introenumerable subset---and this will be the case when we build an introenumerable set in Theorem \ref{intro e neither}.
\begin{lem}[\cite{green}]
If $A\leq^{ui}_T B$ and $B\in\Delta^1_1(A)$, then $B$ has a uniformly introreducible subset, which is also $\Delta^1_1(A)$.
\end{lem}
\begin{prop}[\cite{green}]
Assume $A$ is introenumerable and $A\leq^{ui}_e B$. Then there is $C\in[B]^\omega$ such that $C\in\Delta^1_1(B)$ and $A\leq^{ui}_T C$.
\end{prop}
\begin{cor}\label{hyp cor}
Assume $A$ is introenumerable and there is $A_0\in[A]^\omega$ such that $A_0\in\Delta^1_1(A)$ and $A\leq^{ui}_e A_0$. Then $A$ has a uniformly introreducible subset.
\end{cor}
 \section{Remaining separations}
 We complete the justification for Figure \ref{figure1} in the introduction in this section. Jockusch showed that introreducibility $+$ uniform introenumerability implies uniform introreduciblity \cite[Theorem 5.3]{jock}; we will refer to this result as Jockusch's Intersection Theorem. He claims that a similar thing is true for major-T/e sets. We present a proof for completeness.
 \begin{thrm}\label{majorjock}
 If $A$ is majorreducible and uniformly majorenumerable, then it is uniformly majorreducible.
 \end{thrm}
 \begin{proof}
 We show that if $A$ is uniformly majorenumerable and not uniformly majorreducible, then it is not majorreducible. We will build $g\geq p_A$ such that $g\ngeq_T A$ by building its initial segments. Let $\sigma_0$ be the empty string. Having defined $\sigma_n\in\omega^{<\omega}$, define $\sigma_{n+1}$ as follows:
 \begin{enumerate}
 \item If there is $\sigma\succ\sigma_n$ with $\sigma\geq p_A$ (where $\geq$ is as in Theorem \ref{pi11}) such that $\Phi_n^{\sigma}$ disagrees with $A$ on some input, let $\sigma_{n+1}=\sigma$.
 \item If there is $\sigma\succ\sigma_n$ with $\sigma\geq p_A$ and an $m\in A$ such that no $\tau\succ\sigma$ with $\tau\geq p_A$ have $\Phi_n^\tau(m)\downarrow$, then let $\sigma_{n+1}=\sigma$.
 \end{enumerate}
 We claim that one of the two cases above must happen for all $n$. If $\sigma_{n+1}$ could not have been defined using the above algorithm, then for all $m$, $m\in A$ if and only if there exists $\sigma\succ\sigma_n$ with $\sigma\geq p_A$ such that $\Phi_{n}^\sigma(m)\downarrow=1$. Likewise, if $\sigma_{n+1}$ could not have been defined, then $m\notin A$ if and only if there exists $\sigma\succ\sigma_n$ with $\sigma\geq p_A$ such that $\Phi_n^\sigma(m)\downarrow=0$. Let $\Gamma$ be an enumeration operator witnessing the uniform majorenumerability. Let $g$ be a function majorizing $p_A$. Then for all $m$,
 $$m\in A\liff \exists\sigma,s(\sigma\succ\sigma_n\wedge\sigma\geq p_{\Gamma^g_s}\wedge\Phi_{n,s}^\sigma(m)=1)$$
 and
 $$m\notin A\liff \exists\sigma,s(\sigma\succ\sigma_n\wedge\sigma\geq p_{\Gamma^g_s}\wedge\Phi_{n,s}^\sigma(m)=0)$$
 which shows that $A$ is uniformly majorreducible, contradicting our assumption.
 Having shown the claim, taking $\bigcup_n\sigma_n=G$, this function majorizes $p_A$ and does not compute $A$.
 \end{proof}
 \begin{cor}\label{iT+ume}
 If $A$ is introreducible and uniformly majorenumerable, then it is uniformly majorreducible.
 \end{cor}
 \begin{proof}
 By the theorem, it suffices to show that such $A$ is majorreducible. Given a function $f\geq p_A$, we know that $f$ can compute an infinite subset of $A$ as $f\geq_e A$. By the introreducibility, any infinite subset of $A$ computes $A$.
 \end{proof}
 Moreover, this proof also shows that any majorenumerable set that is introreducible is actually majorreducible.
 \begin{thrm}
 Majorreducible sets are exactly the sets that are both majorenumerable and introreducible.
 \end{thrm}
 \begin{lem}\label{semirecursivemajor}
 Let $A$ be immune and semicomputable\footnote{A set $A$ is semicomputable if there exists a computable function $\alpha$ such that for all $x,y\in\omega$, 1) $\alpha(x,y)\in\{x,y\}$ and 2) if either $x\in A$ or $y\in A$, then $\alpha(x,y)\in A$.}. Then there is an enumeration operator $\Gamma$ such that $\Gamma^f$ is an infinite subset of $A$ for any $f\geq p_A$.
 \end{lem}
 \begin{proof}
 Take $A$ immune and semicomputable, and let $\alpha$ be a partial computable function witnessing the semicomputableness. Let $f$ be a function majorizing $p_A$. Note that $[0,f(n)]\cap A$ contains at least $n+1$ elements. We show that for each $n$, we can uniformly list some $n+1$ elements of $A$. Given $n\in\omega$, we can (in some uniform way) list elements of $[0,f(n)]$ as $x_0,\dots,x_N$ such that $\alpha(x_i,x_{i+1})=x_i$ for $i\leq n-1$. Then $x_0,\dots,x_n$ must be in $A$, as otherwise, there exists $j>n$ such that $x_j\in A$ due to the choice of $f$. Then semicomputableness and the way we listed $x_i$'s imply that $x_0,\dots,x_j\in A$ as of a contradiction.
 \end{proof}
 By \cite[Theorem 5.2]{semirecursive}, any nonzero Turing degree $\leq_T 0'$ contains a set $A$ such that both $A$ and $\ol{A}$ are immune and semicomputable.
 \begin{thrm}[Jockusch]\label{hiuie}
 If $A$ is immune and semicomputable, then $A$ is hyperimmune and uniformly introenumerable.
 \end{thrm}
 \begin{proof}
 Let $A$ be immune and semicomputable. To see that $A$ is hyperimmune, assume not, so that there is a computable function $g$ which majorizes $p_A$. Then by Lemma \ref{semirecursivemajor}, we may let $\Gamma$ be an enumeration operator such that $\Gamma^g$ is an infinite subset of $A$. This contradicts the immunity.
 
 To see that $A$ is uniformly introenumerable, let $C\in[A]^\omega$. For each $a\in A$, note that $\{b\in\omega:\alpha(a,b)=b\}$ must be a finite set by the immunity. Hence, for infinitely many elements $c$ of $C$, it must be that $\alpha(a,c)=a$. Hence, the following uniform procedure enumerates $A$: enumerate $a$ if and only if there exists $c\in C$ such that $\alpha(a,c)=a$.
 \end{proof}
 \begin{cor}
 If $A$ is immune and semicomputable, then $A$ is hyperimmune and uniformly majorenumerable.
 \end{cor}
 \begin{proof}
 By Lemma \ref{semirecursivemajor} and Theorem \ref{hiuie}, it follows that if $A$ is immune and semicomputable, then any $f\geq p_A$ can uniformly enumerate an infinite subset of $A$, which uniformly enumerates $A$.
 \end{proof}
 \begin{cor}[Jockusch]
 There exist complementary hyperimmune uniformly majorenumerable sets.
 \end{cor}
 \begin{cor}\label{ume not umt}
 There exists a uniformly majorenumerable set which is not introreducible, and hence not uniformly majorreducible.
 \end{cor}
 \begin{proof}
 Let $A$ and $\ol{A}$ be both immune and semicomputable sets. Then neither of them are introreducible. For if $A$ (or $\ol{A}$) is introreducible, then $A$ (or $\ol{A}$) would be uniformly majorreducible by Corollary \ref{iT+ume}. Recall that Jockush showed that if $A$ is uniformly introreducible and $\ol{A}$ is uniformly introenumerable, then one of them is c.e.\ \cite[Corollary 3.10]{jock}. So this cannot happen due to the immunity. 
 \end{proof}
 \begin{cor}
 There exists a majorenumerable set which is not majorreducible.
 \end{cor}
 \begin{proof}
 Any uniformly majorenumerable set which is not uniformly majorreducible must be non-majorreducible by Theorem \ref{majorjock}.
 \end{proof}
 \begin{prop}\label{neither}
 There exists a uniformly introenumerable set which is neither introreducible nor regressive.
 \end{prop}
 \begin{proof}
 Let $A$ be a set such that both $A$ and $\ol{A}$ are hyperimmune uniformly introenumerable. Then neither $A$ nor $\ol{A}$ are introreducible. For if one of them is introreducible, then it would be uniformly introreducible by Jockusch's Intersection Theorem. In that case, as in the proof of Corollary \ref{ume not umt}, one of them will be c.e. Recall that if both $A$ and $\ol{A}$ are regressive, then one of them is c.e.\ \cite[Theorem 1]{appel}. This tells that not both $A$ and $\ol{A}$ are regressive. So one of $A$ or $\ol{A}$ is neither introreducible nor regressive.
 \end{proof}
  \begin{prop}\label{unif intro e vs intro T}
 Introreducibility and uniform introenumerability are distinct notions.
 \end{prop}
 \begin{proof}
 If $A$ is an introreducible set that is not uniformly introreducible, then $A$ cannot be uniformly introenumerable by Jockusch's Intersection Theorem.

 And if $A$ is the uniformly introenumerable set from Proposition \ref{neither}, then $A$ is not introreducible.
 \end{proof}
 \begin{thrm}\label{retraceable not umt}
 There exists a retraceable set $A$ such that no $B\in[A]^\omega$ is uniformly majorreducible.
 \end{thrm}
 \begin{proof}
 It suffices to show that there is a retraceable set $A$ without an infinite hyperarithmetic subset by Corollary \ref{hyp maj}. Let $(H_n,n\in\omega)$ be a list of all infinite hyperarithmetic sets. Fix a computable bijection $\pi:2^{<\omega}\to\omega$ such that $\pi(\sigma)>\pi(\tau)$ whenever $|\sigma|>|\tau|$. We build a sequence of strings $\sigma_0\prec\sigma_1\prec\cdots$, and let $A$ be the set $\{\pi(\rho):\exists n\hs \rho\preceq\sigma_n\}$. In this way, $A$ is clearly retraceable because we can consider the map $\omega\to\omega$ by $\pi(\sigma)\to\pi(\sigma^-)$ when $|\sigma|>0$ and $\pi(\sigma)\to\pi(\sigma)$ when $|\sigma|=0$, where $\sigma^-$ denotes the string without the last bit of $\sigma$. Let $\sigma_0=\emptyset$. Having $\sigma_n$, we may pick a string $\tau\in 2^{<\omega}$ such that $|\tau|>|\sigma_n|$ and $\pi(\tau)\in H_n$. Then pick a string $\rho\in 2^{<\omega}$ such that $\rho\succ\sigma_n$ and $\rho$ is incomparable to $\tau$. We define $\sigma_{n+1}=\rho$. Thus, $H_n$ cannot be a subset of $A$ as $\pi(\tau)\notin A$.
 \end{proof}
 Note that the set $A$ above is necessarily not co-immune by Proposition \ref{retraceable implies umt}.
 \begin{cor}\label{retraceable not me}
 There exists a retraceable (and hence uniformly introreducible) set that is not majorenumerable.
 \end{cor}
 \begin{proof}
 By Corollary \ref{refinable}, every majorenumerable set has a uniformly majorreducible subset. So the above retraceable set cannot be majorenumerable.
 \end{proof}
Hence, retraceability, and any other weaker notions like introreducibility, is not refinable to majorenumerability.

 We have seen so far that uniformly major-T/e, major-T/e are all distinct notions and have seen their relations with other notions. Now we focus more towards the notions uniformly intro-T/e, intro-T/e, retraceable/regressive, and c.e./computable sets. To distinguish c.e.\ with introreducibility, note that every Turing degree contains an introreducible set and that every noncomputable c.e.\ set cannot be introreducible.
  \begin{rem}
 For the relation between (uniform) introreducibility and regressive/retraceable, Remark \ref{remark} shows that uniformly introreducible sets are not necessarily refinable to regressive sets.
 \end{rem}
 From Theorem \ref{ienotuie forcing proof}, we know that a nontrivial introenumerable sets exist. We present another method of building a nontrival introenumerable set---the below is an `introenumerable' version of Lachlan's Construction that there exists an introreducible set which is not uniformly introreducible.
 \begin{thrm}\label{intro e neither}
 There exists an introenumerable set that is not uniformly introenumerable and not introreducible.
 \end{thrm}
 \begin{proof}
 Let $B$ be a noncomputable regressive set with a c.e.\ complement, which exists by \cite[Theorem T3]{dekker}. Say $B$ is regressive via a partial computable function $f$, and $\ol{B}$ is the c.e.\ set $W$. We will build $A$ which is introenumerable but not introreducible and not uniformly introenumerable. Let $b^n$ denote the $n$th element of the regressive sequence witnessing the regressiveness of $B$, i.e.\ for all $n$, $f(b^{n+1})=b^n$ and $B=\{b^0,b^1,\dots\}$.
We build $A\leq_e B$ so that we satisfy
\begin{enumerate}
\item $\forall n\hs\<n,\<x,y\>\>\in A\rimp x=b^n$;

\item $\forall n\hs\forall^\infty y\hs[\<n,\<b^n,y\>\>\in A\liff y\in B]$;

\item $\forall n\hs \Gamma_n(A^n)\neq B$ where $A^n$ denotes the set which has its $n$th column equal to that of $A$ and has all the other columns empty. Here, $(\Gamma_n,n\in\omega)$ is a list of all enumeration operators;

\item There is some infinite $S\subseteq A$ such that $A\nleq_T S$.
\end{enumerate}
If we succeed in building such $A$, note that we also have $B\leq_e A$ and $B\leq_T A$. Moreover, such $A$ is introenumerable because given any infinite $C\subseteq A$, $C$ either hits infinitely many columns of $A$ or hits infinitely many elements of a column of $A$. In either cases, $C$ can enumerate $B$, which enumerates $A$. The third condition assures that $A$ is not uniformly introenumerable because for if it were, then any infinite subset of $A$ would uniformly enumerate $B$, hence we have $A\geq^{ui}_{e}B$, which contradicts the third condition. To see that $A$ is not introreducible, this is immediate from the fourth condition.

Below, we write $\Gamma_n(\omega^{[e]}=F)$ to denote the enumeration operator $\Gamma_n$ ran with the oracle where every column is empty other than $e$th column, which is $F$.

If we are given an enumeration of $B$ by $B=\bigcup_s B_s$, note that we can also enumerate $b^0,b^1,\dots$ using the regressiveness of $B$. Hence, to build $A\leq_e B$, it suffices to build $A$ using the enumeration of $B$ specifically given by $B_s=\{b^i:i\leq s\}$.
\vspace{0.3cm}

\noindent\emph{Construction of $A$.} Say we have the enumeration of $B$ by $(B_s=\{b^i:i\leq s\},s\in\omega)$.

We have requirements
$$N_e: \Phi_e^S\neq A$$
which will give us the fourth condition and
$$P_e: \Gamma_e(A^e)\neq B$$
which will give us the third condition. For the $N$ requirements, we build $\sigma_s$ at stage $s$, so that $S:=\bigcup_s\sigma_s$ witnesses $A\nleq_T S$. The requirements are assigned priorities as $P_0>N_0>P_1>\cdots$

\emph{Stage $0$}: Start with $A_0=\emptyset$ and $\sigma_0=\emptyset$.

\emph{Stage $s+1$}: Enumerate into $A_{s+1}$ all unrestricted (what elements are restricted is explained later) elements of the form $\<m,\<b^m,y\>\>$ where $m\leq s$ and $y\in B_{s+1}$. For each $N_e$, $e\leq s$, such that its witness $n_{e,s}$ is undefined, define $n_{e,s}$ to be the smallest element of the form $\<e,\<b^e,y\>\>$ that has not been enumerated into $A_{s+1}$; and restrict $n_{e,s}$ from entering $A$. Then let the highest priority requirement with $e\leq s$ which requires attention act.

\emph{How they act.} We say $N_e$ requires attention if it it not yet satisfied and $\Phi_e^{\sigma_s}(n_{e,s})\downarrow$ where $n_{e,s}$ is the witness for $N_e$ at stage $s+1$. And $N_e$ acts by:
\begin{enumerate}
\item If $\Phi_e^{\sigma_s}(n_{e,s})=0$, enumerate $n_{e,s}$ into $A$.

\item If $\Phi_e^{\sigma_s}(n_{e,s})=1$, keep restricting $n_{e,s}$ from joining $A$.
\end{enumerate}

We say $P_e$ requires attention if it is not yet satisfied and there is some $n\in W_s$ and some finite $F\subseteq \{b^e\}\times\omega$ such that $n\in \Gamma_e(\omega^{[e]}=F)$. And $P_e$ acts by enumerating elements of $F$ into the $e$th column of $A$---note that $P_e$ can override the restrictions set by $N_e$ and that $N_e$ is the only requirement which restricts an element of the $e$th column.

Once $P_e$ acts, it does not act ever again, whereas $N_e$ may get injured (at most once by $P_e$) and act once more. If $P_e$ acted, then clear $n_{e,s}$ and set $n_{i,s+1}=n_{i,s}$ for $i\neq e$ and $i\leq s$. If either $N_e$ acted or no requirement acted, set $n_{i,s+1}=n_{i,s}$ for all $i\leq s$. In either cases, extend $\sigma_s$ to $\sigma_{s+1}$ by enumerating into $\sigma_s$ one element of $A_{s}^s$.
\vspace{0.3cm}

\noindent\emph{Verification.} It is clear that $B\leq_e A$ and $A\leq_e B$. The first condition is satisfied by the construction. As $P$ requirements act at most once and $N$ requirements act at most twice, the second condition is satisfied as well.

\emph{Condition (3) is satisfied.} If $P_e$ ever acts, say by enumerating $F$ into the $e$th column of $A$, then for some $n$, we have $n\in\Gamma_e(\omega^{[e]}=F)\subseteq\Gamma_e(A^e)$ and $n\in \ol{B}$, so that (3) is satisfied. If $P_e$ never acts, we must have $\Gamma_e(A^e)\neq B$ as well. For if $\Gamma_e(A^e)=B$, then $B$ is c.e.\ as of a contradiction because we obtain the following $\Sigma_1^0$ description of $B$:
$$\forall m [m\in B\liff \exists F\subseteq\{b^e\}\times\omega\hs (m\in\Gamma_e(\omega^{[e]}=F))].$$

\emph{Condition (4) is satisfied.} Take $S=\bigcup_s\sigma_s$. We show that $A\nleq_T S$. Assume for contradiction that $A=\Phi_e^S$. Let $s+1$ be a large enough stage such that $N_e$ does not get injured from then on, and hence $n_{e,s}=n_{e,t}$ for $t\geq s$. Since $\Phi^S_e$ is total, $\Phi_e^S(n_{e,s})\downarrow$ at some later stage $t$, and $N_e$ will act at some later stage, succeeding in the diagonalization, which is a contradiction.
 \end{proof}
 Note that the introenumerable sets built in the above proof and in the proof of Theorem \ref{ienotuie forcing proof} are not co-c.e.\ by Corollary \ref{introeimpliesintrot}. An interesting point to note is that the set $A$ constructed above is still a disjoint union of \emph{two} sets $\geq^{ui}_e A$, as in Theorem \ref{ienotuie forcing proof}. Indeed, for each $n\in\omega$, let $x_n$ be a number such that: for $y>x_n$, we have $\<n,\<b^n,y\>\>\in A$ if and only if $y\in B$. In other words, $x_n$'s can be thought of as the numbers noting the injury. Then we can divide $A$ into the two sets as 
 $$A_0=\{\<m,\<b^m,y\>\>:m\in\omega\hs\&\hs y\leq x_m\}$$
 and
 $$A_1=\{\<m,\<b^m,y\>\>:m\in\omega\hs\&\hs y>x_m\}.$$
 Then $B\leq^{ui}_eA_0$ because any infinite subset of $A_0$ must hit infinitely many columns of $A_0$. And $B\leq^{ui}_e A_1$ because although an infinite subset of $A_1$ may hit only finitely many columns of $A_1$, the subset will then contain infinitely many elements from a single column of $A_1$. And any element of a column of $A_1$ is a \emph{true} element by how we chose the $x_n$'s. In particular, any element $\<n,\<b^n,y\>\>\in A_1$ tells us that $b^n\in B$ and $y\in B$, and this shows $B\leq^{ui}_e A_1$. The fact that $A\leq_e B$ implies that $A_0$ and $A_1$ are $\geq^{ui}_e A$.

 Another interesting point to note is that we can choose $x_n$'s so that they are $B'$-computable because the above construction is $B$-computable. Hence, $A_0$ and $A_1$ are $B'$-computable. In particular, $A_0$ is $\Delta^1_1(A)$, and by Corollary \ref{hyp cor}, $A$ has a uniformly introreducible subset.
 
\subsection{Building a nontrivial regressive set} Now, to complete the justification for Figure \ref{figure1} we use a finite injury argument to construct a nontrivial regressive set. The examples of regressive sets we have seen so far are either retraceable or c.e. The following constructs a new regressive set and adds to Theorem \ref{separation}.
 \begin{thrm}\label{regressivenotce}
 There exists a regressive set which is not c.e.\ and not retraceable.
 \end{thrm}
 \begin{proof}
 We build such a set $A$ (which will be $\leq_T 0'$) using a finite injury argument. We will satisfy the requirements
 \begin{align*}
P_e&:\text{if }W_e\text{ has infinitely many even numbers, }W_e\neq A\\
N_e&:A\text{ is not retraceable via }\phii_e
\end{align*}
where the requirements' priorities are assigned as $P_1>N_1>P_2>\cdots$; we are ignoring $P_0,N_0$ without loss of generality.

We will build a regressive set $A$ such that $A$ contains infinitely many even numbers, so that if all $P$ and $N$ requirements are satisfied, then $A$ is as desired. In the below construction, we build a partial computable function $f=\bigcup_sf_s$ and numbers $a_{e,s}$ where in the end, $a_e=\lim_s a_{e,s}$ and $A=\{a_0,a_1,\dots\}$. Some $N_e$ requirements will have a witness $n_{e,s}$ at stage $s$. Each witness, when defined, cannot be removed from $A$ by a lower priority requirement. Moreover, the witnesses will be chosen in a way so that $n_{e,s}-1$ is always not in $A_s$ as long as $N_e$ has not acted---to have the flexibility of throwing $n_{e,s}-1$ into $A_{s+1}$ later if we want.
\vspace{0.3cm}

\noindent\emph{Construction.} Stage $0$: Start with $A_0=\{a_{0,0}=0,a_{1,0}=2\}$ and $f_0=\{(0,0),(2,0)\}$. Let $N_1$ hold the witness $n_{1,0}=2$; let $N_1$ restrict $a_{0,0},a_{1,0}$ from being removed.

Stage $s+1$: Having $A_s=\{a_{0,s},\dots,a_{k,s}\},k\leq s$, and $f_s$, let the highest priority requirement $e\leq s$ which requires attention act; if there is no such requirement, then go to the last paragraph of this construction.

We say $P_e$ requires attention at stage $s+1$ if there is some even number $m\neq 0$ unrestricted by higher priority requirements such that $\phii_e(m)\downarrow$. It acts by restricting $m$ from joining $A$ if $m\notin A_s$. If $a_{i,s}=m$ for some $i$, then $P_e$ acts by clearing $a_{j,s}$ for $j\geq i$, i.e.\ remove them from $A_s$ and undefine them. Then let $P_e$ put a restriction on $m$ from joining $A$.

We say $N_e$ requires attention at stage $s+1$ if $n_{e,s}\in A_s$ is defined and $\phii_e(n_{e,s})\downarrow$. Say $n_{e,s}=a_{l,s}\in A$ for some $l$. Let $a_{m,s}\in A$ be the largest number in $A$ smaller than $a_{l,s}$; by the construction, $a_{m,s}$ exists and $a_{m,s}\leq n_{e,s}-2$. $N_e$ acts as follows:
\begin{itemize}
\item If $\phii_e(a_{l,s})=a_{m,s}$, then note that $a_{l,s}-1$ is unrestricted (as witnesses are always an even number). Undefine $a_{j,s}$ for $j\geq l+1$ and remove them from $A_s$. Redefine $a_{l+1,s}=a_{l,s}-1$ and put $a_{l+1,s}\in A_{s+1}$. Define $f_{s+1}(a_{l+1,s})=a_{l,s}$. Let $N_e$ restrict $a_{0,s},\dots,a_{l+1,s}$ from being removed. Then clear all the witnesses $n_{j,s},j\geq e+1$.

\item If $\phii_{e}(a_{l,s})\neq a_{m,s}$, let $N_e$ keep restricting $a_{l,s}=n_{e,s}$ from being removed. Further, let $N_e$ restrict any $a_{i,s}$ for $i\leq \max\{m,l\}$ from being removed.
\end{itemize}
To finish defining $A_{s+1}$ and $f_{s+1}$, let $j$ be the smallest number such that $a_{j,s}$ is undefined now. Define $a_{j,s}$ to be the smallest even number larger than all numbers mentioned so far, and put it into $A_{s+1}$. Redenote each $a_{j,s}$ in $A_{s+1}$ by $a_{j,s+1}$. Define $f_{s+1}(a_{j,s+1})=a_{j-1,s+1}$. For the numbers of $\dom(f_s)$ that is not yet defined on $f_{s+1}$, define them on $f_{s+1}$ so that $f_{s+1}$ extends $f_s$. Let $i$ be the smallest number such that $N_i$ is unsatisfied yet and does not have its witness. Let $n_{i,s+1}=a_{j,s+1}$, and let $N_i$ restrict $a_{0,s+1},\dots,a_{j,s+1}$ from being removed.
\vspace{0.3cm}

\noindent\emph{Verification.} Clearly, this is a finite injury construction---once a requirement acts, all lower priority requirements are injured, meaning that their restrictions are cleared and their witnesses are cleared. We show that $A$ has infinitely many even numbers and each $a_e:=\lim_s a_{e,s}$ is defined. For $e>0$, let $s+1$ be a large enough stage such that $N_e$ does not get injured from then on. Let $n_{e,s+1}=a_{l,s+1}$ be the witness that $N_e$ holds, which is an even number. Then no matter $N_e$ acts or not from then on, $n_{e,t}=a_{l,t}$ for all $t\geq s$, which shows $a_l=n_e$ is defined and is an even number. And hence, each $a_i,i\leq l$ is also defined. Since $a_i$'s are distinct, we see that $A$ has infinitely many even numbers. We now show that $A$ is not c.e., not retraceable, and is regressive via $f=\bigcup_s f_s$.

\emph{$A$ is not c.e.} For if it were, then $A=W_e$ for some $W_e$ which has infinitely many even numbers. Let $s+1$ be some large enough stage such that all requirements with higher priority than $P_e$ does not act from then on. If $P_e$ ever acts since stage $s+1$, it is clear that we get a contradiction. If $P_e$ does not ever act since stage $s+1$, let $n_{e,s+1}$ be the witness that $N_e$ holds at stage $s+1$. Because $P_e$ does not ever act from then on, we have $n_e=n_{e,s+1}$ and $\phii_e(n_e)\uparrow$, a contradiction.

\emph{$A$ is not retraceable.} Say $A$ were retraceable via $\phii_e$. Let $s+1$ be some large enough stage such that no requirement with higher priority than $N_e$ acts from then on. Let $n_{e,s+1}$ be the witness that $N_e$ holds at stage $s+1$. Then we must have $\phii_e(n_{e,s+1})\downarrow$ at some later stage, as we will have $n_e=n_{e,s+1}\in A$. But then $N_e$ will act, diagonalizing against $\phii_e$ being the retracing function for $A$.

\emph{$A$ is regressive via $f$.} Clearly, $f$ is a partial computable function. So it suffices to show that for all $e+1$, $f(a_{e+1})=a_e$. Note that by the construction, $f(a_{e+1})=a_e$ implies that $f(a_{d+1})=a_d$ for all $d\leq e$. So it suffices to show that $f(a_{e+1})=a_e$ is true for all $a_{e+1}$'s which are $n_{k+1}$ for some $k$. For $k>0$, let $s+1$ be some large enough stage such that no requirement higher priority than $N_k$ acts from then on, and so that $n_{k}=n_{k,s+1}$. Say $n_{k,s+1}=a_{l,s+1}\in A_{s+1}$. No matter if $N_k$ acts or not from then on, we have $f(a_{l,s+1})=a_{l-1,s+1}$, and $a_l=a_{l,s+1}$, $a_{l-1}=a_{l-1,s+1}$ because $N_k$ does not get injured from then on.
\end{proof}
Note that in the proof of Theorem \ref{regressivenotce}, injury could happen from both $P_e$ and $N_e$ requirements. Another thing to note is that the set $A$ constructed in Theorem \ref{regressivenotce} is not retraceable, although it is quite close to being retraceable in the sense that for those numbers $n\in A$ such that $f(n)>n$, it must be that $f(n)=n+1$. In fact, using this, we can show that $A$ is introreducible.
\begin{cor}\label{regressive intro T}
There exists a regressive (uniformly) introreducible set which is not c.e.\ and not retraceable.
\end{cor}
\begin{proof}
We show that the set $A$ constructed in Theorem \ref{regressivenotce} is introreducible. Let $f$ be the function constructed in the proof of Theorem \ref{regressivenotce}. Let $C\in[A]^\omega$. For $n\in\omega$, let $n_0$ be the second smallest number of $C$ larger than $n$. Recall that $\hat{f}(n_0)$ denotes the set $\{m\in\omega:\exists k\hs (m=f^k(n_0))\}$ where $f^k$ means applying $f$ $k$-times. If $n\in \hat{f}(n_0)$, then clearly, $n\in A$. If $n\notin \hat{f}(n_0)$, it is not hard to see that $n\notin A$ using the observation above. So we can compute whether $n\in A$ or $n\notin A$, and the procedure is uniform in $C\in[A]^\omega$.
\end{proof}
To get a regressive set which is not introreducible as well, we need to modify the $N$ requirements, so that the set $A$ we construct is further from being retraceable. In addition, we can require the set to be co-immune.
\begin{thrm}\label{regressive not intro T}
There exists a co-immune regressive set which is not introreducible, not c.e., and not retraceable.
\end{thrm}
\begin{proof}
To ensure that the set we build, which we denote by $B$, is not introreducible, it suffices to ensure that it is not uniformly introreducible by Jockusch's Intersection Theorem; noting that we ensure $B$ is regressive, and hence uniformly introenumerable. Also, $B$ is not retraceable if it is not introreducible.

We modify the $N_e$ requirements from Theorem \ref{regressivenotce} as the following requirements
\begin{align*}
\hat{N}_e&:\text{ if there is a finite }F\subseteq B\text{ such that }\min\{F\}>n_e\text{ and }\Phi_e^F(n_e)\downarrow,\\
&\hspace{0.3cm}\text{ then }B(n_e)\neq\Phi_e^F(n_e)
\end{align*}
where $n_{e}=\lim_s n_{e,s}$ is a witness that we define below for the requirement $\hat{N_e}$. For the co-immunity, as in Proposition \ref{co-immune req}, we have the requirements
$$S_e: \text{if }W_e\text{ is infinite, then }W_e\cap B\neq\emptyset.$$
And as in Theorem \ref{regressivenotce}, to ensure that $B$ is not c.e., we meet the requirements
$$R_e:\text{if }W_e\text{ has infinitely many even numbers, }W_e\neq B.$$
A difference is that, in Theorem \ref{regressivenotce}, once a $P$ requirement acts, it is satisfied permanently, whereas in this proof, $R$ requirements may get injured by a higher priority requirement co-immunity requirement.

The goal is that $\hat{N_e}$ ensures that the set $B$ we build is not uniformly introreducible via $\Phi_e$, where the witness $n_e$ will do the diagonalization against $\Phi_e$. To ensure that $B$ is still regressive, we build a partial computable function $F=\bigcup_sF_s$ and the numbers $b_{e,s}$. In the end, we will have $b_e=\lim_sb_{e,s}$ and $B=\{b_0,b_1,\dots\}$. The requirements $R_e,S_e$ and $\hat{N}_e$'s priorities are assigned as $R_1>S_1>\hat{N}_1>R_2>S_2>\hat{N}_2>\cdots$. We make sure that $b_{0,0}=b_0=0$ and every number $m$ defined in the domain of $F$ is such that $F^n(m)=b_0$ for some $n$.
\vspace{0.3cm}

\noindent\emph{Construction.} Stage $0$: Start with $B_0=\{b_{0,0}=0,b_{1,0}=2\}$ and $F_0=\{(0,0),(2,0)\}$. We consider $0$ being restricted at all times. Let $\hat{N}_1$ hold its witness $n_{1,0}=1$. Each witness an $\hat{N}_e$ requirement holds is restricted from joining $B$ by $\hat{N}_e$ until it acts.

Stage $s+1$: Having $B_s=\{b_{0,s},\dots,b_{k,s}\}$ with $k\leq s$ for some $k$ and having $F_s$, let the highest priority requirement with $e\leq s$ act, following the below instruction. If there is no such requirement, then go to the last paragraph of this construction.

\emph{How they act.} The $R$ requirements are similar to those in Theorem \ref{regressivenotce}. We say $R_e$ requires attention at stage $s+1$ if there is some even number $m$ unrestricted by the higher priority requirements such that $\phii_e(m)\downarrow$. If $R_e$ is the highest priority requirement which requires attention, let it act by restricting $m$ from joining $B$ if $m\notin B_s$. If $b_{i,s}=m$ for some $i$, then $R_e$ acts by undefining $b_{j,s}$ for $i\leq j\leq k$, i.e.\ remove them from $B_s$ and undefine them. Then let $R_e$ put a restriction on $m$ from joining $B$. Further, in the latter case, to ensure that $B$ has infinitely many even numbers, we define $b_{i,s+1}<\cdots<b_{k,s+1}$ to be the even numbers larger than all numbers mentioned so far, and let $R_e$ restrict all of them from being removed. Define $F_{s+1}(b_{k,s+1})=b_{k-1,s+1},\dots,F_{s+1}(b_{i,s+1})=b_{i-1,s}$.

We say $\hat{N}_e$ requires attention at stage $s+1$ if $n_{e,s}$ has been defined previously, and there exists a finite set $F\subseteq B_s$ such that $\min\{F\}>n_{e,s}$ and $\Phi_{e,s}^F(n_{e,s})\downarrow$. When $\hat{N}_e$ is the highest priority requirement which requires attention, it acts as follows:

\begin{enumerate}
\item If $\Phi_e^F(n_{e,s})\downarrow=0$, then define $b_{k+1,s+1}=n_{e,s}$. Then $\hat{N}_e$ requirement puts a restriction on $b_{0,s},\dots,b_{k+1,s+1}$ from being removed. We set $F_{s+1}(b_{k+1,s+1})=b_{k,s}$.

\item If $\Phi_e^F(n_{e,s})\downarrow=1$, then $\hat{N}_e$ requirement restricts $n_{e,s}$ from joining $B$. Also, $\hat{N}_e$ puts a restriction on the elements $b_{0,s},\dots,b_{k,s}$ from being removed---this ensures that the use $F$ is a subset of $B$.
\end{enumerate}
We now describe how the $S$ requirements act. The $S$ requirements are similar to the $Q$ requirements of Proposition \ref{co-immune req}. Just as in Proposition \ref{co-immune req}, at stage $s+1$, we say that a number $x\in \dom(F_s)$ \emph{respects priorities up to $R_e$} if
\begin{itemize}
\item for every number $y$ which is restricted from being removed from $B$ by a requirement of a higher or equal priority than $R_e$ at stage $s+1$, we have $y\in \hat{F_s}(x)$.
\end{itemize}
We say $S_e$ requires attention at stage $s+1$ if there is some $m\in W_{e,s}$ such that $m>n_{i,s}$ for $i<e$ and either
\begin{enumerate}
\item $m$ is not in $\dom(F_s)$ yet, or
\item $m$ is in $\dom(F_s)$ and $m$ respects priorities up to $R_e$.
\end{enumerate}
If $S_e$ is the highest priority requirement which requires attention, in case (1), let it act by setting $b_{k+1,s+1}=m$ and $F_{s+1}(b_{k+1,s+1})=b_{k,s}$. In this case, let $S_e$ restrict $m$ and $b_{i,s}$ for $i\leq k$ from being removed.

In case (2) where we have $m\in \dom(F_s)$, say $l$ is the smallest number such that $F_s^l(m)=0$. Then set $b_{i,s+1}=F_s^{l-i}(m)$ for $0\leq i\leq l$. In addition, if $l<k$, define $b_{l+1,s+1}<\cdots<b_{k,s+1}$ to be the smallest even numbers larger than any number mentioned so far to ensure that $B$ contains infinitely many even numbers. Define $F_{s+1}(b_{k,s+1})=b_{k-1,s+1},\dots,F_{s+1}(b_{1,s+1})=b_{0,s+1}$ and let $S_e$ restrict $b_{i,s+1}$ for $i\leq k$ from being removed.

To finish defining $B_{s+1}$ and $F_{s+1}$, let $j$ be the smallest number such that $b_{j,s}$ is undefined at this point. Define $b_{j,s}$ to be the smallest even number greater than all numbers mentioned so far. Redenote each $b_{i,s}$ by $b_{i,s+1}$ for each $i$ such that $b_{i,s+1}$ is not yet defined. Define $F_{s+1}(b_{j,s+1})=b_{j-1,s+1}$. For the numbers in $\dom(F_s)$ that is not yet defined on $F_{s+1}$, define them on $F_{s+1}$ so that $F_{s+1}$ extends $F_s$. We set $B_{s+1}$ to be the collection $\{b_{0,s+1},\dots,b_{j,s+1}\}$. Let $i$ be the smallest number such that $\hat{N}_i$ is unsatisfied yet and does not have its witness. Let $n_{i,s+1}=b_{j,s+1}-1$.
\vspace{0.3cm}

\noindent\emph{Verification.} Clearly, this is a finite injury construction---once a requirement acts, all lower priority requirements are injured, meaning that their restrictions are cleared, and their witnesses are cleared. We show that each $b_e:=\lim_s b_{e,s}$ is defined and $B$ has infinitely many even numbers.

To see that $b_e=\lim_sb_{e,s}$ exists, let $s$ be a stage where $b_{e,s}$ is defined for the first time. Note that $\hat{N}$ requirements never remove or undefine $b_{e,s}$. If no $R$ and $S$ requirements remove $b_{e,s}$ from $B$, then $b_e=b_{e,s}$. If at some stage $t>s$, some $R_i$ (or $S_i$) requirement removes $b_{e,s}$ from $B_t$, then let $T>t$ be a large enough stage such that no requirement of a higher priority than that of $R_i$ (or $S_i$) act from stage $T$. Then $b_e=b_{e,T}$ because $b_{e,T}$ is restricted by some requirement of higher or equal priority than that of $R_i$ (or $S_i$) from being removed: note that from stage $t$, whenever $b_{e,-}$ changes, the new $b_{e,-}$ is restricted from being removed by some higher priority requirement. To see that each $n_e=\lim_s n_{e,s}$ exists, let $s$ be a stage where $\hat{N}_e$ is never injured from then on. Then $n_e=n_{e,s}$.

To see that $B$ has infinitely many even numbers, note that in the construction, whenever a requirement acts and removes some even numbers from $B$, it also adds new even numbers into $B$ and restricts them from being removed. We now show that $B$ is not c.e., not introreducible, co-immune, and regressive via $F=\bigcup_s F_s$.

\emph{$B$ is not c.e.} For if it were, then $B=W_e$ for some $W_e$ which has infinitely many even numbers. Let $s+1$ be some large enough stage such that all requirements of a higher priority than that of $R_e$ do not act from then on. If $R_e$ ever acts since stage $s+1$, it is clear that we get a contradiction. If $R_e$ does not ever act since stage $s+1$, we may let $M\in B$ be a large enough even number such that has never been mentioned by stage $s+1$. Because $R_e$ does not ever act since stage $s+1$, we have $\phii_e(M)\uparrow$, which is a contradiction.

\emph{$B$ is regressive via $F=\bigcup_sF_s$.} Clearly, $F$ is a partial computable function. So it suffices to show that for all $e+1$, $F(b_{e+1})=b_e$. Note that by the construction, for all $e,s$, whenever $b_{e+1,s}$ is defined, we have $F_s(b_{e+1,s})=b_{e,s}$. Further, when $e,s$ are such that $b_{e+1,t}=b_{e+1,s}=b_{e+1}$ for all $t\geq s$, it must also be that $b_{e,t}=b_{e,s}=b_e$ for all $t\geq s$, as the only way the value of $b_{e,s}$ changes is that $b_{e,s}$ gets removed, which would also remove $b_{e+1,s}$. Hence, let $s$ be a large enough stage such that $b_{e+1,s}=b_{e+1,t}$ for all $t\geq s$. Then $b_{e,s}=b_e$, and we have $F_s(b_{e+1})=F_s(b_{e+1,s})=b_{e,s}=b_e$. Since $F$ extends $F_s$, this shows that $F$ is a regressing function for $B$.

\emph{$\hat{A}$ is co-immune.} We show that $\ol{B}$ does not contain $W_e$ for each $e$ such that $W_e$ is infinite. If there is some stage $s$ such that $S_e$ acts and is never injured from then on, then it is clear that $W_e\nsubseteq\ol{B}$. If there is no such stage, there must be some stage $s$ such that $S_e$ is unsatisfied and never acts from then on. Let $t>s$ be a large enough stage such that no requirement of a higher priority than that of $S_e$ acts from stage $t$. For if $W_e\subseteq\ol{B}$ and $W_e$ is infinite, then we may let $x\in W_e$ be a number larger than any number mentioned by stage $t$ of the construction. Let $t_0\geq t$ be a stage such that $\phii_{e,t_0}(x)\downarrow$. Then no matter $x$ is in or not in $\dom(F_{t_0})$, $S_e$ requires attention and acts at stage $t_0$---if $x\in \dom(F_{t_0})$, this implies that some requirement put $x$ into the domain after stage $t$, so that $x$ necessarily respects priorities up to $R_e$. This is a contradiction.

\emph{$B$ is not introreducible.} As noted earlier, it suffices to show that $B$ is not uniformly introreducible via $\Phi_e$ for each $e\in\omega$. Assume $B$ is uniformly introreducible via $\Phi_e$ for a contradiction. Let $s$ be a large enough stage so that $n_e=n_{e,s}$ and no higher priority requirement than that of $\hat{N}_e$ ever acts from then on. Note that there is a finite set $F\subseteq B$ such that $\min\{F\}>n_e$ and $\Phi^F_e(n_e)\downarrow$ by the uniform introreducibility---otherwise, $\{k\in B:k>n_e\}$ is an infinite subset of $B$ which does not witness that `$B$ is uniformly introreducible via $\Phi_e$'. But then we can find some large enough stage $s_0>s$ such that $F\subseteq B_{s_0}$, $\min\{F\}>n_e$, and $\Phi_{e,s_0}^F(n_e)\downarrow$. Then $\hat{N}_e$ requirement will act at that stage, diagonalizing against $\Phi_e$.
\end{proof}
\nocite{*}
\bibliographystyle{alpha}
\bibliography{sample}
\end{document}